\documentclass[12pt,a4paper]{article}

\usepackage{dsfont}
\usepackage{amssymb}
\usepackage{amsthm}
\usepackage{amsfonts}
\usepackage{amsmath}
\usepackage{dsfont}
\usepackage{anysize}
\usepackage{hyperref}
\usepackage{bbm}
\usepackage{setspace}
\usepackage{color}
\usepackage[margin=1in]{geometry}
\usepackage[style=alphabetic,maxnames=99,maxalphanames=5, isbn=false, giveninits=true, doi=false, url=false]{biblatex}

 \AtBeginBibliography{\footnotesize}
 
\renewbibmacro{in:}{}

\DeclareFieldFormat[article]{citetitle}{#1}
\DeclareFieldFormat[article]{title}{#1}
\DeclareFieldFormat[inbook]{citetitle}{#1}
\DeclareFieldFormat[inbook]{title}{#1}
\DeclareFieldFormat[incollection]{citetitle}{#1}
\DeclareFieldFormat[incollection]{title}{#1}
\DeclareFieldFormat[inproceedings]{citetitle}{#1}
\DeclareFieldFormat[inproceedings]{title}{#1}
\DeclareFieldFormat[phdthesis]{citetitle}{#1}
\DeclareFieldFormat[phdthesis]{title}{#1}
\DeclareFieldFormat[misc]{citetitle}{#1}
\DeclareFieldFormat[misc]{title}{#1}
\DeclareFieldFormat[book]{citetitle}{#1}
\DeclareFieldFormat[book]{title}{#1} 

\newtheorem*{theorem*}{Theorem}

\newtheorem{theorem}{Theorem}[section]

\newtheorem{lemma}[theorem]{Lemma}
\newtheorem{corollary}[theorem]{Corollary}
\newtheorem{conjecture}[theorem]{Conjecture}

\newtheorem{remark}[theorem]{Remark}

\newtheorem{definition}[theorem]{Definition}
\newtheorem{proposition}[theorem]{Proposition}

\newcommand{\RR}{\mathbb{R}}

\newcommand{\gf}{\mathrm{gf}}
\newcommand{\Div}{\mathrm{div}}

\DeclareMathOperator{\Ric}{Ric}

\def\cF{\mathcal{F}}
\def\NN{\mathbb{N}}
\def\RR{\mathbb{R}}

\DeclareMathOperator{\ord}{ord}

\newcommand{\PSH}{\mathrm{PSH}}
\usepackage{mathtools}

\DeclareMathOperator{\Spec}{Spec}

\DeclareMathOperator{\vol}{vol}
\DeclareMathOperator{\usc}{usc}

\newcommand{\ddc}{\mathrm{dd}^{\mathrm{c}}}

\newcommand{\triv}{\mathrm{triv}}
\newcommand{\NA}{\mathrm{NA}}
\newcommand{\An}{\mathrm{an}}

\newcommand{\NS}{\mathrm{NS}}
\newcommand{\val}{\mathrm{val}}

\def\beq{\begin{equation}}
\def\eeq{\end{equation}}

\bibliography{ComplexGeometry}

\title{A transcendental approach to non-Archimedean metrics of pseudoeffective classes}
\author{Tam\'as Darvas, Mingchen Xia and Kewei Zhang}
\date{}
\begin{document}
\maketitle
\begin{abstract} We introduce the concept of non-Archimedean metrics attached to a transcendental pseudoeffective cohomology class on a compact K\"ahler manifold. This is obtained via extending the Ross--Witt Nystr\"om correspondence to the relative case, and we point out that our construction agrees with that of Boucksom--Jonsson when the class is induced by a pseudoeffective $\mathbb Q$-line bundle. 

We introduce the notion of a flag configuration attached to a transcendental big class, recovering the notion of a test configuration in the ample case. We show that non-Archimedean finite energy metrics are approximable by flag configurations, and very general versions of the radial Ding energy are continuous, a novel result even in the ample case. As applications, we characterize the delta invariant as the Ding semistability threshold of flag configurations and filtrations, and prove a YTD type existence theorem for K\"ahler--Einstein metrics in terms of flag configurations.
\end{abstract}
\tableofcontents

\section{Introduction and results}

In recent years, especially since the appearance of \cite{BBJ21}, the non-Archimedean approach to K-stability has gained  attention, as evidenced by recent works \cite{BJ18b, BE21, BHJ17, BJ18,  DeLe20, Li20,BJ22b}, to name only a few papers in the fast expanding literature.

Studying K-stability for degenerate classes is a natural extension of the much studied ample case \cite{CDS15, DS16a,  Tian15, TiWa20,CSW,BBJ21,Zh21}. Recent works aim to understand existence of canonical metrics in this context as well. Degenerate K\"ahler--Einstein metrics have taken the forefront \cite{LTW21a,LTW21b,LXZ22, DZ22, DeRe22, Xu22}, but one would think that the notion of constant scalar curvature K\"ahler metrics also extends to big classes, as minimzers of the K-energy functional. With the study still in its infancy \cite{KZ22,DNL22}, a better understanding of K-stability of big classes is needed before substantial progress can be made. Our work tries to fill some of this void.

We propose a transcendental approach to non-Archimedean metrics, that allows for their treatment even for a transcendental pseudoeffective class. Our analytic treatment has a number of advantages. The envelope conjecture of Boucksom--Jonsson naturally holds in our context, immediately recovering the main result of \cite{BJ22}, as a particular case. We show that very general versions of the radial Ding energy are continuous, and that finite energy non-Archimedean metrics can be approximated by algebraic objects called flag configurations (that reduce to test configurations in the ample case). The latter two properties allow to characterize the analytically defined delta invariant using algebraic data, and provide a YTD type existence theorem for big K\"ahler--Einstein (KE) metrics, without using the minimal model program, connecting with recent work of the first and third authors \cite{DZ22}.\vspace{-0.2cm}

\paragraph{Transcendental non-Archimedean metrics.}To state our main results, we fix some terminology. Let $X$ be a compact K\"ahler manifold and $\theta$ a smooth closed real $(1,1)$-form representing a pseudoeffective cohomology class $\{\theta\}$. We denote by $\textup{PSH}(X,\theta)$ the space of $\theta$-psh functions, and all complex Monge--Amp\`ere measures in this work will be assumed to be pluripolar in the sense of \cite{BEGZ10}.

A relative test curve is a map $\mathbb R \ni \tau \mapsto \psi_\tau \in \textup{PSH}(X,\theta)$ that is $\tau$-decreasing, $\tau$-concave, and $\tau$-usc. Moreover, $\psi_\tau \equiv -\infty$ for all $\tau$ big enough, and the limit $\psi_{-\infty} = \lim_{\tau \to -\infty} \psi_\tau \in \textup{PSH}(X,\theta)$ exists. 
Such an object will be denoted $\{\psi_\tau\}_\tau$, and we say that $\{\psi_\tau\}_\tau$ is a test curve relative to $\psi_{-\infty}$.

The relative test curve $\{\psi_\tau\}_\tau$ is $\mathcal I$-maximal if $\psi_\tau = P[\psi_\tau]_\mathcal I$ for all $\tau \in \mathbb R$, where $P[\cdot ]_\mathcal I$ is the following envelope:
\[
P[u]_\mathcal I := \sup\{ w \in\textup{ PSH}(X, \theta), w \leq 0, \mathcal I(tw) \subseteq \mathcal I(tu), t \geq 0\} . 
\]
Here $\mathcal I (tu)$ is the multiplier ideal sheaf of $u$, locally generated by holomorphic functions $f$ such that
$|f|^2 e^{-tu}$ is integrable.

To exclude pathological behaviour, we need to avoid vanishing mass. For any big class $\{\theta\}$, by $\textup{MTC}_\mathcal I(X,\theta)$ we denote the set of maximal test curves that additionally satisfy $\int_X \theta_{\psi_{-\infty}}^n > 0$. As we will see in \eqref{eq: transition_map}, given any K\"ahler metric $\omega$ from the K\"ahler cone $\mathcal K$, there is a natural map 
$$\textup{MTC}_\mathcal I(X,\theta) \to \textup{MTC}_\mathcal I(X,\theta + \omega)$$
This allows to define the space of non-Archimedean metrics associated with a pseudoeffective class $\{\theta\}$ as the following projective limit in the category of sets:
\[
\PSH^{\NA}(\theta):=\varprojlim_{\omega\in \mathcal{K}}\textup{MTC}_\mathcal I(X, \theta+\omega).
\]

 We refer to Definition \ref{def:PSH-an} for the precise details.
Given $f \in C^\infty(X)$, there is an obvious identification between $\PSH^{\NA}(\theta)$ and $\PSH^{\NA}(\theta+\ddc f)$ allowing us to regard $\PSH^{\NA}(\theta)$ as an invariant of the cohomology class $\{\theta\}$.

When $\{\theta\}$ is the first Chern class of a $\mathbb Q$-line bundle $L$, Boucksom and Jonsson defined the space of non-Archimedean metrics $\textup{PSH}^\NA(L^\An)$ in \cite{BJ18b} using algebraic tools (see Section \ref{subsec:NAformal}). In our first result we show that our construction agrees with theirs in this important particular case.

\begin{theorem}(=Theorem \ref{thm:NAtoNAbij})\label{mainthm: isomorphism} When $\{\theta\}$ is the first Chern class of a pseudoeffective $\mathbb Q$-line bundle $L$, there is a natural bijection between $
\PSH^{\NA}(\theta)$ and $\textup{PSH}(L^\An)$.
\end{theorem}
The same holds for a pseudoeffective $(1,1)$-class lying in the Néron--Severi group with real coefficients, as we explain in Remark~\ref{rmk:realclass_iso}.

For the construction of the natural map between  $\PSH^{\NA}(\theta)$ and $\textup{PSH}(L^\An)$, we refer to \eqref{eq: BJ_map_def}. 
Recently Boucksom--Jonsson have showed that the so-called envelope conjecture holds in case $\{\theta\}$ is the first Chern class of a $\mathbb Q$-line bundle $L$ \cite{BJ22}. We will show that the same property holds in the transcendental case as well, thus recovering the Boucksom--Jonsson result as a particular case, via the above isomorphism theorem. Let us note here an equivalent formulation of the envelope conjecture, which is easy to state:

\begin{theorem}(=Theorem~\ref{thm:increasingnettestcurveconv}) Suppose that $\{\theta\}$ is a pseudoeffective class. Any bounded from above increasing net of $\PSH^{\NA}(\theta)$  has a supremum inside $\PSH^{\NA}(\theta)$.
\end{theorem}

We refer to Theorem~\ref{thm:increasingnettestcurveconv}  for the precise statement, and see Conjecture~\ref{conj:env} for the equivalent form of the envelope conjecture in non-Archimedean geometry. Moreover, in Corollary \ref{cor:env_conj} we explain how the above result yields \cite[Theorem A]{BJ22}.

For general transcendental classes $\{\theta\}$ we cannot interpret $\PSH^{\NA}(\theta)$ as metrics on certain Berkovich spaces, but an analogous study is still possible, and we will treat this thoroughly in a separate paper.\vspace{-0.2cm}

\paragraph{Approximation of finite energy non-Archimedean metrics.} Let us assume for the rest of this introduction that $\{\theta\}$ is a big class.
A maximal test curve $\{\psi_{\tau}\}_\tau$ is of  \emph{finite energy} (notation: $\{\psi_{\tau}\}_\tau \in \mathcal R^1(X,\theta)$) if

    \begin{equation*}\label{eq: main_fetestcurve_def}
   \int_{-\infty}^{\tau^+_\psi} \left( \int_X \theta_{\psi_\tau}^n-\int_X \theta_{V_\theta}^n \right) \,\mathrm{d}\tau >-\infty\,,
    \end{equation*}
    where $V_{\theta}\in \PSH(X,\theta)$ is the potential with minimal singularity. 
Such test curves are in a one-to-one correspondence with finite energy geodesic rays, per the Ross--Witt Nystr\"om correspondence, elaborated in \cite{DZ22,DL20} (see Theorem~\ref{thm:RWNpresc}(iii) and  \eqref{eq:radial-I-formula}). 
As proved in \cite[Theorem 2.14]{DDNL5} there is a \emph{chordal} metric $d_1^c$ on $\mathcal R^1(X,\theta)$ making it a complete metric space (see Section 2.2).
    
We identify the space of finite energy non-Archimedean metrics $\mathcal E^{1,\NA}(X,\theta)$ with $\mathcal R^1_\mathcal I(X,\theta)$, the space of $\mathcal I$-maximal finite energy test curves. The definition of $\mathcal E^{1,\NA}(X,\theta)$ agrees with that of Boucksom--Jonsson in the ample case, and one can approximate any finite energy non-Archimedean metric using test configurations, by \cite[Theorem 1.1]{DX20}. We show in this paper that the analogous result holds in the big case as well.

To state the result, let us define a \emph{flag configuration} of a big cohomology class $\{\theta\}$ to be a (partial) flag of coherent analytic ideal sheaves
\[
\mathfrak{a}_0 \subseteq \mathfrak{a}_1 \subseteq \dots \subseteq \mathfrak{a}_N  = \mathcal{I}(V_\theta).
\]
By convention, $\mathfrak{a}_{\ell}:=0$ for $\ell\in \mathbb{Z}_{<0}$ and $\mathfrak{a}_{\ell}:=\mathfrak{a}_N$ if $\ell\in \mathbb{Z}_{\geq N}$. A flag configuration will be conveniently denoted as an analytic coherent ideal sheaf on the product $X \times \mathbb C$: 
\[
\mathfrak a := \mathfrak a_0 + \mathfrak a_1 s + \dots +\mathfrak a_{N-1} s^{N-1} + \mathfrak a_N (s^N) \subseteq \mathcal O(X \times \mathbb C),
\]
where $s$ denotes the coordinate on $\mathbb{C}=\Spec \mathbb{C}[s]$. In this transcendental big setting, we avoid calling flag configurations actual test configurations, to avoid confusion with the concepts introduced in \cite{DeRe22} for big line bundles,  which are not yet proved to be equivalent with ours. In addition, in the K\"ahler case, there is notion of transcendental test configuration introduced in \cite{DerRos17,Dyr16}. Further investigations are needed to find the correct analogue of this notion is in the big case, and prove possible equivalency with our flag configurations.

As we point out in \eqref{eq: filtrate_flag_def}, in case $\{\theta\}$ is the first Chern class of a big line bundle $L$, to a flag configuration one can associate a natural filtration $\mathcal F^\mathfrak a$ of the section ring $R(X,L)$. Extending a construction of Phong--Sturm/Ross--Witt Nystr\"om \cite{PS07, RWN14}, we show that even in the big case, one can associate a natural geodesic ray/maximal test curve to any filtration (see Definition \ref{def:PSray1}).

Next we show that finite energy non-Archimedean metrics can be approximated by flag configurations or filtrations, extending  \cite[Theorem 1.1]{DX20} to the big case, making contact with the ideas of \cite[Section 5]{BBJ21}. 

\begin{theorem}(=Theorem \ref{prop:BBJmaxappbyflagconf})\label{mainthm: approx}  When $\{\theta\}$ is the first Chern class of a big line bundle $L$,  elements of $\mathcal R^1_\mathcal I(X,\theta) = \mathcal E^{1,\NA}(X,L^\An)$ can be $d_1^c$--approximated by Phong--Sturm rays of flag configurations.
\end{theorem}

As far as we are aware, approximations results of this nature have not been explored in the big case before. However in the K\"ahler case, perhaps the first result of this nature appeared in \cite[Corollary 1.3]{Berm19trans-Bergman}, where Berman devised an approximation scheme for geodesic rays coming from deformations to the normal cone. Another related result from the K\"ahler case is \cite[Theorem 1.5]{DL20}. With the usual difficulties associated with big classes in place, techniques from the K\"ahler case don't seem to translate to our setting.

\paragraph{Continuity of the radial Ding functional.} Staying with an arbitrary big class $\{\theta\}$, we consider qpsh functions  $\psi,\chi$ on $X$, with $\chi$ having analytic singularity type. After adding some constants to either $\chi$ and $\psi$, 
 one can attach to $(\chi,\psi)$ a Radon probability measure, following \cite{BBEGZ16}:
\begin{equation}
    \label{eq:def-mu}
    \mu := e^{\chi - \psi} \omega^n.
\end{equation}

Next, following \cite{Ding88}, one defines the $\lambda$-Ding functional for $\lambda >0$. This is $\mathcal D_{\mu}^\lambda: \mathcal E^1(X,\theta) \to \RR$, the $\lambda$-Ding functional:
\begin{flalign}\label{def: twisted_Ding}
\mathcal D_\mu^\lambda(\varphi)=-\frac{1}{\lambda}\log\int_Xe^{-\lambda\varphi} \mathrm{d}\mu-I_\theta(\varphi)\text{ for }\varphi\in\mathcal E^1(X,\theta),
\end{flalign}
where $I_\theta(\cdot)$ is the Monge-Amp\`ere energy. For the above definition to make sense, we  assume that $c_\mu[V_\theta]:=\sup\{\gamma \geq 0: \int_X e^{-\gamma V_\theta} \,\mathrm{d}\mu <\infty\}> \lambda$.

The Euler--Lagrange equation of the $\lambda$-Ding functional 
is the following twisted Monge--Amp\`ere equation \cite{DZ22}:
\begin{equation}\label{eq: KE_cont equation}
(\theta + \ddc u)^n  = e^{-\lambda u + \chi - \psi} \omega^n.
\end{equation}

Solutions to this equation, represent potentials along the continuity method for a twisted KE metric, that solves the above equation in the particular case $\lambda = 1$ \cite{DZ22}.

Despite lack of convexity of $\mathcal D^\lambda_\mu$,  \cite[Theorem 1.4]{DZ22} gives a formula for the slope of the $\lambda$-Ding functional along subgeodesic rays. Let $(0,\infty) \ni t \mapsto u_t \in \mathcal E^1(X,\theta)$ be a sublinear subgeodesic ray \cite[\S 3]{DZ22}. Then
\begin{equation}\label{eq: Ding_slope}
\varliminf_{t \to \infty} \frac{\mathcal  D_{\mu}^\lambda(u_t)}{t} = -\lim_{t \to \infty} \frac{I_\theta(u_t)}{t} + \sup\{\tau \ \in \RR \ : \  c_\mu[\hat u_\tau] \geq \lambda \},
\end{equation}
where $c_\mu[u] := \sup\{\gamma \geq 0 : \int_X e^{-\gamma u} \,\mathrm{d}\mu <\infty\}$. 
On the heels of the above result it is convenient to introduce  
\[
\mathcal D_{\mu}^\lambda \{u_t\} := \varliminf_{t \to \infty} \frac{\mathcal  D_{\mu}^\lambda(u_t)}{t},
\]
and we will call this expression the \emph{radial $\lambda$-Ding functional} of the subgeodesic $\{u_t\}_t$. This should be viewed as the transcendental analogue of the $\beta_\delta$-functional in \cite[\S 4]{XZ20}.

In our next main result we show that the 
 radial $\lambda$-Ding functional is $d_1^c$-continuous.

\begin{theorem}(=Theorem \ref{thm:YYTD})\label{mainthm: Ding_cont} For any $\lambda\in (0,c_{\mu}[V_{\theta}])$, the functional $\mathcal{D}_{\mu}^\lambda: \mathcal{R}^1(X,\theta) \to \mathbb{R}$ is continuous.
\end{theorem}
In case of numerical classes, there is a precise estimate due to Boucksom--Jonsson  for the non-Archimedean Ding functional \cite[Proposition~3.15]{BJ22b}, implying continuity. We wonder if  this estimate can be extended to our radial transcendental setting.

\paragraph{Stability and KE metrics.}We discuss applications to delta invariants and KE metrics for a big class $\{\theta\}$. 
Based on the Blum--Jonsson interpretation of the Fujita--Odaka delta invariant \cite{BJ17, BJ18,FO18}, rooted in the non-Archimedean approach to K-stability (see \cite[Definition 7.2]{BBJ21}), we recall the definition of the twisted delta invariant of our data \cite{DZ22}:
\begin{equation}\label{eq: delta_psi_def}
\delta_\mu(\{\theta\}):=\inf_E\frac{A_{\chi,\psi}(E)}{S_\theta(E)}.
\end{equation}
Here the infimum is taken over all prime divisors $E$ over $X$, i.e., $E$ is a prime divisor on a modification $\pi: Y \to X$ of $X$ (cf. \cite[\S B.5]{BBJ21}). 
Also, 
$A_{\chi,\psi}(E):=A_X(E)+\nu(\chi,E)-\nu(\psi,E)$, $A_X(E)$ is the log discrepancy of $E$
and $\nu(\psi,E)$ denotes the Lelong number of $\pi^*\psi$ along $E$ (cf. \cite[(13)]{DZ22}).
The \emph{expected Lelong number} $S_\theta(E)$ of $\{\theta\}$ along $E$ is defined by
\begin{equation*}\label{eq: SE_def}
S_\theta(E):=\frac{1}{\vol(\{\theta\})}\int_0^{\tau_\theta(E)}\vol(\{\pi^*\theta\}-x\{E\})\,\mathrm{d}x,
\end{equation*}
where $\tau_\theta(E) := \sup\{\tau \in \RR : \{\pi^*\theta\}-\tau\{E\} \textup{ is big}\}$ is the pseudoeffective threshold, and the volume function $\vol(\cdot)$ is understood in the sense of \cite{BEGZ10}. 

In \cite[Theorem 1.5]{DZ22} it was shown that $\delta_\mu$ can be computed as the geodesic semistability threshold of the $\lambda$-Ding functionals. In our next main result we additionally argue that $\delta_\mu$ can be computed as the non-Archimedean Ding semistability threshold as well. For the definition of the non-Archimedean Ding functional $\mathcal{D}^{\lambda,\NA}_\mu$ we refer to  \eqref{eq:Lreformulate}.

\begin{theorem}(=Theorem \ref{thm: delta_I_maximal_semistab}, Corollary \ref{cor: NAcor}) \label{mthm: delta} We have the following identities:
\begin{flalign}\label{meq: delta_geod_semi_stab}
\delta_\mu&=\sup\{\lambda>0 \ | \ \mathcal{D}_{\mu}^\lambda\{u_t\}\geq 0\text{ for all sublinear subgeodesic ray }u_t\in\mathcal{E}^1(X,\theta)\}\\
&=\sup\left\{\lambda>0 : \mathcal{D}_{\mu}^\lambda\{u_t\}\geq 0, \ \{u_t\}_t \in \mathcal R^1(X,\theta)\right\} \nonumber\\
&=\sup\left\{\lambda>0 : \mathcal{D}_{\mu}^\lambda\{u_t\}\geq 0, \ \{u_t\}_t \in \mathcal R^1_{\mathcal{I}}(X,\theta)\right\} \nonumber
\end{flalign}
When $\{\theta\}=c_1(L)$ for some big line bundle $L$ on $X$, we further have
\begin{flalign}
\delta_\mu&=\sup\left\{\lambda>0 : \mathcal{D}_{\mu}^\lambda\{u_t\}\geq 0, \ \forall  \textup{ rays }\{u_t\}_t  \textup{ induced by filtrations}\right\} \nonumber\\
&=\sup\left\{\lambda>0 : \mathcal{D}_{\mu}^\lambda\{u_t\}\geq 0, \ \forall  \textup{ rays }\{u_t\}_t  \textup{ induced by flag configurations}\right\} \nonumber \\
&=\sup\left\{\lambda>0 : \mathcal{D}_{\mu}^{\lambda,\NA}(u) \geq 0, \ u \in \mathcal E^{1,\NA}(X,\theta)\right\} .\nonumber
\end{flalign}
\end{theorem}

Lastly, we turn to stability and assume that $\{\theta\} = c_1(-K_X)$. For simplicity we only treat the untwisted case, where $\psi =0$ and $\chi :=f \in C^\infty(X)$ satisfying $\theta + \ddc f = \Ric \omega$.

\begin{definition}
\label{def:-K-flag-Ding-stable}
    We say that $(X,-K_X)$ is \emph{uniformly Ding stable with respect to flag configurations}, if there exists $\varepsilon >0$ such that
\begin{equation}\label{eq: uniform_Ding_stab_flag_def_intr}
\mathcal{D}_\mu^{1,\NA} (u)\geq \varepsilon \mathcal J^{\NA}(u),
\end{equation}
for any $u = \{u_t\}_t \in \mathcal E^{1,\NA}(X,\theta) $ induced by flag configurations.
\end{definition}
  For the definition of the non-Archimedean $J$ functional $\mathcal{J}^{\NA}_\mu$ we refer to  \eqref{eq:Jreformulate}. We refer to Section \ref{sec: YTD} for more details, and specifically to \eqref{eq: Ding_algebraic} for an algebraic/valuative interpretation of uniform Ding stability.

Combining Theorems \ref{mainthm: approx}, \ref{mainthm: Ding_cont} and \ref{mthm: delta} with the results of \cite{DZ22}, we prove a uniform YTD type existence theorem for KE metrics.

\begin{theorem}(=Theorem \ref{thm:YYTD})
\label{thm:YTD}
Suppose $-K_X$ is big. If $(X,-K_X)$ is uniformly Ding stable with respect to flag configurations then there exists a KE metric. Specifically, there exists $ u \in \textup{PSH}(X,\theta)$ having minimal singularity type such that $\theta_u = \Ric \theta_u$, i.e., 
$$
\theta_u^n = e^{f- u} \omega^n.
$$\end{theorem}

It follows from \cite[Remark 1.3]{Xu22} and the main result of \cite{LTW21b}  that finite generation of the anticanonical section ring and uniform Ding stability with respect to test configurations (as defined by Dervan--Reboulet \cite{DeRe22} in the big case) also imply existence of KE metrics (for details see \cite[Section 6.2]{DZ22}). One advantage of our Theorem \ref{thm:YTD} is that we don't impose any conditions on finite generation (or $-K_X$ being klt) apriori. Similar to \cite{DZ22}, our approach is purely analytical, but on the heels of \cite{Xu22} it is natural to ask if one could give another proof of Theorem \ref{thm:YTD} using techniques of the minimal model program.

Around the same time our paper appeared on the arXiv, the very intriguing work \cite{Tr23} was published, proving a YTD type existence theorem for K\"ahler--Einstein metrics with prescribed singularity type. With a different motivation, Trusiani independently developed the theory of relative test curves, overlapping with our results in Section 2.2. 

\paragraph{Acknowledgments.}\vspace{-0.3cm} We would like to thank Ruadhai Dervan, Mattias Jonsson, Yaxiong Liu and Pietro Piccione for their comments. The first named author was partially supported by an Alfred P. Sloan Fellowship and National Science Foundation grant DMS--1846942. The second named author is supported by Knut och Alice Wallenbergs Stiftelse grant KAW 2021.0231. The third named  author is supported by NSFC grants 12101052, 12271040, and the Fundamental Research Funds 2021NTST10.

\section{Preliminaries}

\subsection{Finite energy pluripotential theory}\label{subsec:finitee}

We give a very brief account of finite energy pluripotential theory in the big case. For a more complete treatment we refer to the papers \cite{BEGZ10}, \cite{DDNL3}, \cite[\S 3]{DX21}, \cite[\S 3]{DZ22} and the textbook \cite{GZ17}.

\paragraph{Finite energy pluripotential theory.} Let $(X,\omega)$ be a connected compact K\"ahler manifold of dimension $n$ and  $\theta$ a smooth closed real $(1,1)$-form. A function $u: X \rightarrow \mathbb{R}\cup \{-\infty\}$ is called quasi-plurisubharmonic (qpsh) if locally $u= \rho + \varphi$, where $\rho$ is smooth and $\varphi$ is a plurisubharmonic (psh) function. We say that $u$ is $\theta$-plurisubharmonic ($\theta$-psh) if it is qpsh and $\theta_u:=\theta+\ddc u \geq 0$ as currents, with $\ddc=\sqrt{-1}\partial\bar\partial/2\pi$. We let $\PSH(X,\theta)$ denote the space of $\theta$-psh functions on $X$. 

The class $\{\theta\}$ is {\it big} if there exists $\psi\in \textup{PSH}(X,\theta)$ satisfying $\theta_\psi\geq \varepsilon \omega$ for some $\varepsilon>0$. The class $\{\theta\}$ is \emph{pseudoeffective} if $\{\theta + \varepsilon \omega\}$ is big for any $\varepsilon >0$.
We assume that  $\{\theta\}$ is pseudoeffective in this section, unless specified otherwise.

Given $u,v \in \textup{PSH}(X,\theta)$,  $u$ is more singular than $v$, (notation: $u \preceq v$) if there exists $C\in \mathbb{R}$  such that $u\leq v+C$. The potential $u$ has the same singularity as $v$ (notation: $u \simeq v$) if $u\preceq v$ and $v\preceq u$. 
The classes $[u]$ of this latter equivalence relation are called \emph{singularity types}. When $\{\theta\}$ is merely big, all elements of $\textup{PSH}(X,\theta)$ are very singular, and we distinguish the potential with minimal singularity:
\begin{equation}\label{eq:def-V-theta}
    V_\theta := \sup \{u \in \textup{PSH}(X,\theta) : u \leq 0\}.
\end{equation}
A function $u\in \PSH(X,\theta)$ is said to have minimal singularity if it has the same singularity type as $V_{\theta}$, i.e., $[u]=[V_\theta]$.

We say that $[u]$ is an \emph{analytic singularity type} if it has a representative $u \in \textup{PSH}(X,\theta)$ that locally can be written as $u = c\log (\sum_{j} |f_j|^2) +g$, where $c>0$, $g$ is a bounded function and the $f_j$ are a finite set of holomorphic functions.
By a Demailly's approximation theorem there are plenty of $\theta$-psh functions with analytic singularity type if $\{\theta\}$ is big \cite[Section 14]{Dem12}. 

In \cite{BEGZ10} the authors introduce the non-pluripolar Borel measure $\theta_u^n$ of an element $u \in \textup{PSH}(X,\theta)$, satisfying $\int_X \theta_u^n \leq  \int_X \theta_{V_\theta}^n = : \vol(\{\theta\})$. It was proved in \cite[Theorem 1.2]{WN19} that for any $u,v \in \textup{PSH}(X,\theta)$  we have the following monotonicity result for the masses:
if $v \preceq u$ then $\int_X \theta_v^n \leq \int_X \theta_u^n.$

We say that $u \in \textup{PSH}(X,\theta)$ is a full mass potential ($u \in \mathcal E(X,\theta)$) if 
$\int_X \theta_u^n = \int_X \theta_{V_\theta}^n.$
Moreover, we say that $u \in \mathcal E(X,\theta)$ has finite energy ($u \in \mathcal E^1(X,\theta)$) if
$\int_X |u - V_\theta| \theta_{u}^n < \infty.$

The class $\mathcal E^1(X,\theta)$  plays a central role in the variational theory of complex Monge--Amp\`ere equations, as detailed in \cite{BBEGZ16,BEGZ10} and later works. Here we only mention that the Monge--Amp\`ere energy $I_\theta$ naturally extends to this space with the usual formula:
\begin{equation}\label{eq: I_def}
I_\theta(u) = \frac{1}{\vol(\{\theta\}) (n+1)}\sum_{j = 0}^n \int_X (u - V_\theta) \,\theta_u^j  \wedge \theta_{V_\theta}^{n-j}, \ \ \ \ u \in \mathcal E^1(X,\theta).
\end{equation}
It is upper semi-continuous (usc) with respect to the $L^1$ topology on $\textup{PSH}(X,\theta)$.

Given any $f : X \to [-\infty,+\infty]$ one can consider the envelope $P_\theta(f):=\textup{usc}(\sup\{v \in \textup{PSH}(X,\theta), \ v \leq f \})$, where $\mathrm{usc}$ denotes the upper-semicontinuous regularization. Then, for $u,v \in \textup{PSH}(X,\theta)$ we can introduce the \emph{rooftop envelope} $P_\theta(u,v):=P_\theta(\min(u,v))$.

With the help of these envelopes one can define a complete metric on $\mathcal E^1(X,\theta)$. Indeed, as pointed out in \cite[Theorem 2.10]{DDNL1}, for $u,v \in \mathcal E^1(X,\theta)$ we have that $P(u,v) \in \mathcal E^1(X,\theta)$ and the following expression defines a complete metric on $\mathcal E^1(X,\theta)$ \cite[Theorem 1.1]{DDNL3}:
$$d_1(u,v) = I_\theta(u) + I_\theta(v) - I_\theta(P(u,v)).$$
In addition, $d_1$-convergence implies $L^1$-convergence of qpsh potentials \cite[Theorem 3.11]{DDNL3}. Moreover, any two points of $u_0,u_1 \in \mathcal E^1(X,\theta)$ can be connected by a special $d_1$-geodesic $[0,1] \ni t\mapsto u_t \in \mathcal E^1(X,\theta)$, called the finite energy geodesic segment between $u_0,u_1$.

\paragraph{Envelopes of singularity type.}
We now discuss envelopes attached to singularity types, going back to  \cite{RWN14} in the K\"ahler case: given $u,v\in \PSH(X,\theta)$,
\[
P_\theta[u](v) := \textup{usc}\Big(\lim_{C \to +\infty}P_\theta(u+C,v)\Big)=\textup{usc}\big(\sup\{ w \in\textup{ PSH}(X, \theta): w \leq v, \ [w] \preceq [u]\}\big) .
\]

It is easy to see that $P_\theta[u](v)$ depends only on the singularity type $[u]$. When $v = V_\theta$, we simply write $P[u]:=P_\theta[u]:=P_\theta[u](V_\theta)$ and call this potential the \emph{envelope of the singularity type} $[u]$. It follows from \cite[Theorem 3.8]{DDNL2} that $\theta_{P[u]}^n \leq \mathbbm{1}_{\{P[u] =0\}} \theta^n$. 
Also, by \cite[Proposition 2.3 and Remark 2.5]{DDNL2} we have that $\int_X \theta_{P[u]}^n = \int_X \theta_u^n$.

An algebraic counterpart of $P[u](v)$ is the operator $P[u]_\mathcal I(v) \in \textup{PSH}(X,\theta)$ defined by 
\[
P[u]_\mathcal I(v) := \sup\left\{ w \in\textup{ PSH}(X, \theta): w \leq v, \mathcal I(tw) \subseteq \mathcal I(tu), t \geq 0\right\} . 
\]
Here $\mathcal I (tu)$ is a multiplier ideal sheaf, locally generated by holomorphic functions $f$ such that
$|f|^2 e^{-tu}$ is integrable. This envelope is implicit in \cite{KS20}, and was introduced and studied in detail in
\cite[\S 2.4]{DX20}, \cite{Tr20}, \cite{DX21}. Again, $P[u]_\mathcal I:=P_\theta[u]_\mathcal I(V_\theta)$, and it is not difficult to see that $P[u]_\mathcal I \geq P[u]$ for any $u \in \textup{PSH}(X,\theta)$.

We say $u\in\textup{PSH}(X,\theta)$ is a model potential if $P_\theta[u]=u$;  it is an $\mathcal{I}$-model potential if $P_\theta[u]_\mathcal{I}=u$. By \cite[Proposition 2.18(i)]{DX20}, an $\mathcal{I}$-model potential is always a model potential.

We note the following technical result about model potentials that will be needed later.

\begin{lemma}\label{lma:Imodeldeclelong}
Let $\varphi^j\in \PSH(X,\theta)$ be a decreasing sequence of model potentials. Let $\varphi^j \searrow  \varphi \in \textup{PSH}(X,\theta)$. If $\varphi$ has positive mass then for any prime divisor $E$ over $X$,
    \[
    \lim_{j\to\infty}\nu(\varphi^j,E)=\nu(\varphi,E).
    \]
\end{lemma}
\begin{proof} 
By \cite[Proposition 4.6]{DDNL5} we have that $\int_X \theta_{\varphi^j}^n \searrow \int_X \theta_{\varphi}^n.$ Hence, by \cite[Lemma~4.3]{DDNL5}, for any $\varepsilon \in (0,1)$ and $j$ big enough there exists $\psi^j \in \PSH(X,\theta)$ such that $(1-\varepsilon) \varphi^j + \varepsilon \psi^j \leq \varphi$. This implies that for $j$ big enough we have 
\[
(1-\varepsilon)\nu(\varphi^j,E) + \varepsilon\nu(\psi^j,E) \geq \nu(\varphi,E) \geq \nu(\varphi^j,E)\,.
\]

However, for $E$ fixed, $\nu(\chi,E)$ is uniformly bounded (see, e.g., \cite[Lemma 2.10]{Bo02}) for any $\chi \in \PSH(X,\theta)$. So letting $\varepsilon \searrow 0$ we conclude. 
\end{proof}

Finally, we state an effective version of Guan--Zhou's strong openness theorem that will be used multiple times in this work (see \cite{GuZh15}, \cite{BBJ21}, c.f.  \cite[Theorem 2.2]{DZ22}). 

\begin{theorem}
\label{thm:openness}
    Suppose that there are qpsh functions on $X$ such that $u_j\nearrow u$ a.e. If $f\in\mathcal{I}(u)$, then $f\in\mathcal{I}(u_j)$ for all $j$ big enough. More generally, if $\chi$ is a  qpsh function with analytic singularities satisfying $e^{\chi-u}\in L^1(X)$, then $e^{\chi-u_j}\in L^1(X)$ for all $j$ big enough.
\end{theorem}

\subsection{The theory of relative test curves}

Let $X$ be a connected compact Kähler manifold of dimension $n$. Let $\theta$ be a closed smooth real $(1,1)$-form on $X$ representing a big cohomology class $\{\theta\}$. Set $V=\int_X \theta_{V_{\theta}}^n$.

\paragraph{Test curves.}
A map $\mathbb{R}\ni \tau \mapsto \psi_\tau\in \PSH(X, \theta)$ is a \emph{relative test curve}, denoted $\{\psi_\tau\}_{\tau}$, if $\tau\mapsto \psi_\tau(x)$ is concave, decreasing and usc for any $x\in X$. Moreover, $\psi_\tau\equiv -\infty$ for all $\tau$ big enough and the weak $L^1$ limit $\psi_{-\infty} := \lim_{\tau \to -\infty} \psi_\tau \in \textup{PSH}(X,\theta)$ exists. We say that $\{\psi_\tau\}_\tau$ is a relative test curve relative to $\psi_{-\infty}$.
If $\psi_{-\infty}=V_\theta$ we simply call $\{\psi_\tau\}_\tau$ a test curve, and this particular case was studied in detail in \cite{DDNL3,DZ22}, having its origins in \cite{RWN14}, in the ample case.

The above definition allows to introduce the following constant:
\begin{equation}\label{eq: psi+_tau_def}
\tau_\psi^+ := \inf\{\tau \in \mathbb{R} : \psi_\tau \equiv -\infty\}\,.
\end{equation}

Using the convention $P[-\infty]=-\infty$, a relative test curve $\{\psi_\tau \}_\tau$  can have the following properties: \vspace{0.1cm}\\
(i) $\{\psi_{\tau}\}_\tau$ is \emph{maximal} if $P[\psi_\tau] =\psi_\tau$ for any $\tau \in \mathbb{R}$.\vspace{0.1cm}\\
(ii) A relative test curve $\{\psi_{\tau}\}_\tau$ is a \emph{finite energy test curve} if
    \begin{equation}\label{eq: fetestcurve_def}
     I_{\theta}\{\psi_{\tau}\}:=\tau^+_{\psi}+\frac{1}{V}\int_{-\infty}^{\tau^+_\psi} \left( \int_X \theta_{\psi_\tau}^n-\int_X \theta_{V_\theta}^n \right) \,\mathrm{d}\tau >-\infty\,.
    \end{equation}
(iii) We say that $\{\psi_{\tau}\}_\tau$ is \emph{bounded} if $\psi_\tau=\psi_{-\infty}$ for all $\tau$ small enough.\vspace{0.1cm}\\
(iv) $\{\psi_{\tau}\}_\tau$ is $\mathcal I$-\emph{maximal} if $P[\psi_\tau]_\mathcal I =\psi_\tau$ for any $\tau \in \mathbb{R}$. Since $\psi_\tau \leq P[\psi_\tau] \leq P[\psi_\tau]_\mathcal I$, we notice that $\{\psi_{\tau}\}_\tau$ is maximal if it is $\mathcal I$-maximal.

\paragraph{Subgeodesic rays.}
A \emph{psh subgeodesic ray}, denoted $\{u_t\}_t$, is an assignment $(0,\infty)\ni t\mapsto u_t\in \PSH(X,\theta)$ such that 
\begin{equation}\label{eq: Phi_def_subgeod}
\Phi(x, \xi) := u_{-\log |\xi|^2}(x)
\end{equation}
is $p_1^*\theta$-plurisubharmonic on $X\times \Delta^*$, where $\Delta \subset \mathbb C$ is the unit disk and $\Delta^*=\Delta\setminus\{0\}$, and $p_1:X\times \Delta\rightarrow X$ is the projection.

\smallskip A \emph{sublinear subgeodesic ray} is a psh subgeodesic ray $\{u_t\}_t$ such that the weak $L^1$ limit $u_0 := \lim_{t \to 0+}u_t \in \textup{PSH}(X,\theta)$ exists and for some $C \in \mathbb R$ we have
$$u_t \leq u_0 + Ct.$$
In this case we say that $\{u_t\}_t$ is a sublinear subgeodesic \emph{emanating from} $u_0$.

\smallskip A \emph{psh geodesic segment} in $\PSH(X,\theta)$ from $\varphi$ to $\psi$ is a map $[a,b]\ni t\mapsto u_t\in \PSH(X,\theta)$ such that $\Phi(x,z) := u_{-\log |z|^2}(x)$ is $\pi^*\theta$-psh on  $X\times \{z\in \mathbb{C}:\exp(-b)<|z|^2<\exp(-a)\}$. Moreover, for any $S^1$-invariant $p_1^*\theta$-psh function $\Psi$ on $X\times \{z\in \mathbb{C}:\exp(-b)<|z|^2<\exp(-a)\}$ such that 
\begin{equation}\label{eq: psh_segm_comp}
    \varlimsup_{t\to b-} \Psi(\cdot,\mathrm{e}^{-b})\leq \psi,\quad \varlimsup_{t\to a+} \Psi(\cdot,\mathrm{e}^{-a})\leq \varphi
\end{equation}
    almost everywhere, we have $\Psi\leq \Phi$.
Furthermore, $\Phi$ assumes the correct boundary values: 
    \[
    \lim_{t\to b-} \Phi(\cdot,\mathrm{e}^{-b})=\psi,\quad \lim_{t\to a+} \Phi(\cdot,\mathrm{e}^{-a})=\varphi.
    \]

A \emph{psh geodesic ray} is a sublinear subgeodesic ray $t \mapsto u_t$ such that $[a,b] \ni t \mapsto u_t \in \textup{PSH}(X,\theta)$ is a psh geodesic segment for all $a,b \in (0,\infty)$.

\paragraph{The geometry of finite energy rays}

A \emph{finite energy geodesic ray} $[0,\infty) \ni t \mapsto u_t \in \mathcal E^1(X,\theta)$ is simply a psh geodesic ray with $u_0 = V_\theta$. The space of such finite energy rays, denoted $\mathcal R^1(X,\theta)$, was studied in \cite{DDNL5} (see \cite{DL20} for the K\"ahler case).  

Using linearity/convexity one can define the \emph{chordal} metric between two rays:
\[
d_1^c(\{u_t\}_t, \{v_t\}_t) = \lim_{t \to \infty} \frac{d_1(u_t,v_t)}{t}, \quad \{u_t\}_t, \{v_t\}_t \in \mathcal R^1(X,\theta).
\]
It was shown in \cite[Theorem 2.14]{DDNL5} that $(\mathcal R^1(X,\theta),d_1^c)$ is a complete metric space.
Given any ray $\{u_t\}_t\in \mathcal{R}^1(X,\theta)$, we define its radial energy as
\begin{equation}\label{eq: I_rad_def}
    I_{\theta}\{u_t\}:=\lim_{t\to\infty} \frac{I_{\theta}(u_t)}{t}.
\end{equation}
The limit exists as $I_{\theta}(u_t)$ is linear in $t$.

As in the case of the metric space structure of $(\mathcal E^1(X,\theta),d_1)$, for any rays $\{u_t\}_t,\{v_t\}_t \in \mathcal R^1(X,\theta)$ one can define $\{u_t \land v_t\}_t \in \mathcal R^1(X,\theta)$ (resp. $\{u_t \lor v_t\}_t \in \mathcal R^1(X,\theta)$), the biggest (resp. smallest) ray satisfying $u_t \wedge v_t \leq \min (u_t,v_t)$ (resp. $u_t \wedge v_t \geq \max (u_t,v_t)$) for all $t\geq 0$.

These two special rays satisfy the Pythagorean formula and Pythagorean inequality  respectively (\cite[Theorem~1.3]{Xi19}, \cite[Proposition~2.15]{DDNL5}), for $C = C(n)$:

\begin{equation}\label{eq: Pyth_id}
\begin{aligned}
d_1^c(\{u_t\}_t,\{v_t\}_t)=& d_1^c(\{u_t\}_t,\{u_t \wedge v_t\}_t) + d_1^c(\{u_t \wedge v_t\}_t,\{v_t\}_t),\\
C d_1^c(\{u_t\}_t,\{v_t\}_t) \geq & d_1^c(\{u_t\}_t,\{u_t \vee v_t\}_t) + d_1^c(\{u_t \vee v_t\}_t,\{v_t\}_t).
\end{aligned}
\end{equation}
We note the following result, that is the radial analogue of \cite[Proposition~2.6]{BDL17}:

\begin{proposition}\label{prop: conv_monotone} 
Let $\{u^k_t\}_t,\{u_t\} \in \mathcal R^1(X,\theta)$ such that $d_1^c(\{u^k_t\}_t,\{u_t\}_t) \to 0$. Then there exists a subsequence, again denoted $\{u^k_t\}_t$ and $\{v^k_t\}_t,\{w^k_t\}_t \in \mathcal R^1(X,\theta)$ such that:\\
\noindent (i) $w^k_t \leq u^k_t \leq v^k_t$ for any fixed $t \geq 0$, \\
\noindent (ii) $w^k_t \nearrow u_t$ and $v^k_t \searrow u_t$ for any fixed $t \geq 0$,\\
\noindent (iii) $d_1^c(\{w^k_t\}_t,\{u_t\}_t) \to 0$ and $d_1^c(\{v^k_t\}_t,\{u_t\}_t) \to 0$. 
\end{proposition}

\begin{proof} The proof follows the argument of   \cite[Proposition~2.6]{BDL17} in case of potentials. 
In fact, the sequence $\{w^k_t\}$ is constructed using the exact same ideas, with \eqref{eq: Pyth_id} used in the appropriate places.

The sequence of $\{v^k_t\}$ is constructed following the argument of \cite[Proposition~4.2]{DDNL5} word-for-word, using the radial Pythagorean inequality \eqref{eq: Pyth_id} in the appropriate places.
\end{proof}

\paragraph{The Ross--Witt Nystr\"om correspondence.}
We adopt the following convention: relative test curves will always be parametrized by $\tau$, whereas rays will be parametrized by $t$. Hence, $\{\psi_t\}_t$ will always refer to some kind of ray, whereas  $\{\phi_\tau\}_\tau$ will refer to some type of test curve. As we will now point out, rays and test curves are dual to each other, so one should think of the  parameters $t$ and $\tau$ to be dual to each other as well.

To any sublinear subgeodesic ray $\{\phi_t\}_t$ (relative test curve $\{\psi_\tau\}_\tau$) we can associate  its (inverse) partial Legendre transform at $x \in X$ as 
\begin{equation}\label{eq: Leg_transf_def_ray_test_curve}
    \begin{aligned}
\hat \phi_\tau(x) := \inf_{t>0}(\phi_t(x) - t\tau)\,,& \quad \tau \in \mathbb R\,,\\
\check \psi_t(x) := \sup_{\tau \in \mathbb R}(\psi_\tau(x) + t\tau)\,,& \quad t> 0\,.
\end{aligned}
\end{equation}

We say that a ray $\{u_t\}_t \in \mathcal R^1(X,\theta)$ is $\mathcal I$\emph{-maximal} if $\{\hat u_\tau\}_\tau$ is $\mathcal I$-maximal. The set of $\mathcal{I}$-maximal finite energy rays is denoted by
$\mathcal R^1_\mathcal I(X,\theta)$.
As we will see later, there is a bijection between $\mathcal R^1_\mathcal I(X,\theta)$ and the set of non-Archimedean finite energy potentials.

\begin{lemma}\label{lma:testcurvlegusc}
Let $\{\psi_{\tau}\}_\tau$ be a test curve relative to $\phi$. Then $\sup_{\tau}(\psi_{\tau}(x)+t\tau)$ is usc with respect to $(t,x)\in (0,\infty)\times X$. In particular, $\check{\psi}_t\in \PSH(X,\theta)$ for all $t>0$, and $\{\check \psi_t\}_t$ is a sublinear subgeodesic ray, emanating from $\phi$.
\end{lemma}
\begin{proof}
The argument is exactly the same as in the particular case when $\phi=V_{\theta}$, proved in \cite[Proposition~3.6]{DZ22}.
\end{proof}

We note the following simple result, that will have important consequences in the sequel.

\begin{lemma}\label{lem:Cproj}
Let $\{\psi_{\tau}\}_\tau$ be a test curve relative to $\phi$, with $\int_X \theta_\phi^n>0$.  Then $P[\check{\psi}_t]=P[\phi]$ for all $t>0$.
\end{lemma}
\begin{proof}
We may assume that $\phi\leq 0$, and we observe that $\psi_{\tau}+t\tau \leq \check{\psi}_t\leq \phi+t\tau^+_\psi$
for all $\tau<\tau^+_\psi$ and $t > 0$. Taking envelopes with respect to singularity type, we find
\begin{equation}\label{eq:Ctripineq}
P[\psi_{\tau}]\leq P[\check{\psi}_t]\leq P[\phi].
\end{equation}
But for almost all $x\in X$,
$
\phi(x)=\lim_{\tau\to-\infty}\psi_{\tau}(x)\leq \lim_{\tau\to-\infty}P[\psi_{\tau}].
$
It follows that for almost all $x\in X$,
$
\phi(x)\leq P[\check{\psi}_t](x).
$
As both sides are $\theta$-psh, it follows that the inequality holds everywhere. As both sides have positive mass, we get that $P[\phi]\leq P[\check{\psi}_t]$ due to \cite[Theorem 3.12]{DDNL2}. Together with \eqref{eq:Ctripineq}, we conclude.
\end{proof}

Lastly we state the versions of the Ross--Witt Nystr\"om correspondence between test curves and rays that will be needed in this work:

\begin{theorem}\label{thm:RWNpresc}
Let $\phi \in \textup{PSH}(X,\theta)$ with $\int_X \theta_\phi^n>0$. The following hold: \vspace{0.1cm}\\
\noindent (i) The map $\{\psi_\tau\} \mapsto \{\check \psi_t\}_t$ gives a bijection from the set of test curves relative to $\phi$ to the set of sublinear subgeodesic rays in $\PSH(X,\theta)$ emanating from $\phi$, with inverse $\{u_t\}_t \mapsto \{\hat u_\tau\}_\tau$. \vspace{0.1cm}\\ 
\noindent(ii) 
The map of (i) gives a bijection between the set of maximal test curves relative to $\phi$, and the set of psh geodesic rays in  $\PSH(X,\theta)$ emanating from $\phi$. \vspace{0.1cm}\\
\noindent  (iii) The map of (i) gives a bijection between the set of finite energy maximal test curves, and  the set of finite energy psh geodesic rays $\mathcal R^1(X,\theta)$. 
In addition, under this correspondence, the radial $I_{\theta}$ functional \eqref{eq: I_rad_def} is equal to the energy  of the test curve \eqref{eq: fetestcurve_def}.
\end{theorem}

\begin{proof}
Using Proposition~\ref{lma:testcurvlegusc}, the proof of the first and second part are carried out in exactly the same manner as the particular case $\phi=V_{\theta}$  \cite[Theorem~3.7(i)(ii)]{DZ22}. 
The last part of the theorem is simply  \cite[Theorem~3.9]{DZ22}.
\end{proof}

Test curves relative to potentials $\phi$ with $\int_X \theta_\phi^n=0$ exhibit pathological behaviour (c.f. \cite[Section 4]{DDNL5}). To exclude these we introduce:
\begin{equation}\label{eq: MTC_def}
\textup{MTC}(X,\theta) = \left\{\{\phi_\tau\}_\tau \subset \textup{PSH}(X,\theta)\textup{ relative test with } \int_X \phi_{-\infty}^n>0\right\},
\end{equation}
with ``MTC" indicating \emph{maximal test curve}. 
Similarly, the test curves in $\textup{MTC}(X,\theta)$ that are $\mathcal I$-maximal will be denoted by $\textup{MTC}_\mathcal I(X,\theta)$. For now, we record the following simple observation, that already uses the positive mass assumption crucially.

\begin{lemma}Suppose that $\{\psi_\tau\}_\tau \in \textup{MTC}(X,\theta)$. Then $\psi_{-\infty} = P[\psi_{-\infty}]$.
\end{lemma}

\begin{proof} We have that $\psi_\tau \nearrow \psi_{-\infty}$ a.e., as $\tau \to -\infty$ and $\psi_\tau = P[\psi_\tau]$. By \cite[Corollary~4.7]{DDNL5} we conclude that $\psi_{-\infty} = P[\psi_{-\infty}]$.
\end{proof}

The next two results show $\mathcal I$-maximal test curves are preserved under monotone limits.

\begin{lemma}\label{lma:decnetna}
Let $\{\psi^i_{\tau}\}_\tau\in \textup{MTC}_\mathcal I(\theta)$ be a decreasing net (in the sense that $\psi^i_{\tau}$ is decreasing for each $\tau$). Let $\psi_{\tau}:=\inf_i \psi^i_{\tau}$. 
Assume that for some $\tau \in \mathbb R$, $\psi_{\tau}$ is not identically $-\infty$, and it has positive mass. Then $\{\psi_{\tau}\}_\tau \in \textup{MTC}_\mathcal I(\theta)$.
\end{lemma}
\begin{proof} Since $\psi_{-\infty} \geq \psi_\tau$, we get that $\psi_{-\infty}$ has positive mass by \cite[Theorem 1.2]{WN19}. 
It suffices to observe that $\psi_\tau$ is $\mathcal{I}$-model whenever it is not $-\infty$. Indeed, 
$
\psi_{\tau}\leq P_{\theta}[\psi_{\tau}]_{\mathcal{I}}\leq\inf_i P_{\theta}[\psi^i_{\tau}]_{\mathcal{I}}=\inf_i \psi^i_{\tau}=\psi_{\tau}.
$
\end{proof}

\begin{lemma}\label{lma:incfamily}
Let $\{\psi^i_{\tau}\}_\tau\in \textup{MTC}_\mathcal I(\theta)$ be an increasing net in $i$.  Assume that $\tau^+_\psi:=\sup_i\tau^{i,+}_\psi<\infty$. Let $\psi_{\tau}:=\textup{usc}\big(\sup_i \psi^i_{\tau}\big)$ for $\tau \neq \tau^+_\psi$ and $\psi_{\tau^+_\psi}:=\lim_{\tau\nearrow\tau^+_\psi}\psi_\tau$. 
Then $\{\psi_{\tau}\}_\tau\in \textup{MTC}_\mathcal I(\theta)$.
\end{lemma}

\begin{proof}
Showing that $\psi_{\tau}$ is $\mathcal I$-model and has positive mass for $\tau < \tau^+_\psi$ is a consequence of 
\cite[Theorem 1.2]{WN19} and \cite[Lemma 2.21(iii)]{DX20}. That $\psi_{\tau^+_\psi}$ is $\mathcal I$-model follows from \cite[Lemma 2.21(i)]{DX20}. 
\end{proof}

\paragraph{Maximization.}
We adapt the \emph{maximization} process of test curves from the works \cite{RWN14,DDNL3, DZ22} to our relative setting. Let $\{\psi_\tau\}_\tau$ be a relative test curve with $\int_X \theta_{\psi_{-\infty}}^n>0$. The maximization of $\{\psi_\tau\}_\tau$ is simply the relative test curve $\{\phi_\tau\}_\tau$ such that
\begin{equation}\label{eq: maximization_def}
    \phi_{\tau}:=\left\{
        \begin{aligned}
            P[\psi_\tau],& \quad \text{if }\tau<\tau^+_{\psi};\\
            \lim_{\tau\nearrow\tau^+_{\psi}}\phi_\tau, & \quad \text{if } \tau=\tau^+_{\psi};\\
            -\infty, & \quad \text{if } \tau>\tau^+_{\psi}.
        \end{aligned}
    \right.
\end{equation}
The condition $\int_X \theta_{\psi_{-\infty}}^n>0$, \cite[Theorem~1.2]{WN19} and $\tau$-concavity of $\tau \mapsto \psi_\tau$ implies that $\int_X \theta_{\psi_{\tau}}^n>0$ for all $\tau < \tau^+_\psi$.
By \cite[Theorem~3.12]{DDNL2} we obtain that $\{\phi_\tau\}_\tau$ is a maximal test curve, relative to $\phi_{-\infty} = P[\psi_{-\infty}]$ (by \cite[Corollary~4.7]{DDNL5}).

Lastly, along the same lines as above, we introduce the $\mathcal I$-\emph{maximization} of a  relative test curve $\{\psi_\tau\}_\tau$ with $\int_X \theta_{\psi_{-\infty}}^n>0$. The $\mathcal I $-maximization of $\{\psi_\tau\}_\tau$ is the relative test curve $\{\chi_\tau\}_\tau$ defined by:
\begin{equation}\label{eq: I-maximization_def}
    \chi_{\tau}:=\left\{
        \begin{aligned}
            P[\psi_\tau]_{\mathcal{I}},& \quad \text{if }\tau<\tau^+_{\psi};\\
            \lim_{\tau\nearrow\tau^+_{\psi}}\chi_\tau, & \quad \text{if } \tau=\tau^+_{\psi};\\
            -\infty, & \quad \text{if } \tau>\tau^+_{\psi}.
        \end{aligned}
    \right.
\end{equation}
By definition, $\{\chi_\tau\}_\tau$ is an $\mathcal I$-maximal test curve relative to $\chi_{-\infty} = P[\psi_{-\infty}]_\mathcal I $ by \cite[Lemma~2.21(iii)]{DX20}. 

The maximization/$\mathcal I$-maximization of a sublinear (sub)geodesic $\{u_t\}_t$ will simply be the inverse Legendre transform of the maximization/$\mathcal I$-maximization of $\{\hat u_\tau\}_\tau$.

\subsection{The non-Archimedean formalism of Boucksom--Jonsson}\label{subsec:NAformal}
In this section, we recall the basics of the space of non-Archimedean psh metrics, as defined by Boucksom--Jonsson \cite{BJ18b}.

\paragraph{The space of valuations.}

Let $X$ be an irreducible reduced variety over $\mathbb{C}$ of dimension $n$. We recall the notion of Berkovich analytification $X^{\An}$ of $X$ with respect to the trivial valuation on $\mathbb{C}$. Recall that a (real-valued) valuation on $X$ (or a valuation of $\mathbb{C}(X)$) is a map $v:\mathbb{C}(X)\rightarrow (-\infty,\infty]$ satisfying \vspace{0.2cm}\\
\noindent (i) For $f\in \mathbb{C}(X)$, $v(f)=\infty$ if and only if $f=0$.\vspace{0.2cm}\\
\noindent (ii) For $f,g\in \mathbb{C}(X)$, $v(fg)=v(f)+v(g)$.\vspace{0.2cm}\\
\noindent (iii) For $f,g\in \mathbb{C}(X)$, $v(f+g)\geq \min\{v(f),v(g)\}$.\vspace{0.2cm}\\
The set of valuations on $X$ is denoted by $X^{\val}$. The center of a valuation $v$ is the scheme-theoretic point $c=c(v)$ of $X$ such that $v\geq 0$ on $\mathcal{O}_{X,c}$ and $v>0$ on the maximal ideal $\mathfrak{m}_{X,c}$ of $\mathcal{O}_{X,c}$. The center is unique if exists. It exists if $X$ is proper.

In the remaining of this section, we assume that $X$ is projective.

As a set, $X^{\An}$ is the set of semi-valuations on $X$, in other words, real-valued valuations $v$
on irreducible reduced subvarieties $Y$ in $X$ that is trivial on $\mathbb{C}$. We call $Y$ the \emph{support} of the semi-valuation $v$.
In other words, 
\[
X^{\An}=\coprod_{Y} Y^{\val}.
\]
We will write $v_{\triv}\in X^{\An}$ for the trivial valuation on $X$: $v_{\triv}(f)=0$ for any $f\in \mathbb{C}(X)^{\times}$.
We endow $X^{\An}$ with the coarsest topology such that \vspace{0.2cm}:\\
\noindent (i) For any Zariski open subset $U\subseteq X$, the subset $U^{\An}$ of $X^{\An}$ consisting of semi-valuations whose supports meet $U$ is open.\vspace{0.2cm}\\
\noindent (ii) For each Zariski open subset $U\subseteq X$ and each $f\in H^0(U,\mathcal{O}_X)$ (here $\mathcal{O}_X$ is the sheaf of regular functions), the map $|f|:U^{\An}\rightarrow \mathbb{R}$ sending $v$ to $\exp(-v(f))$ is continuous.\vspace{0.2cm}\\
See \cite{Berk93} for more details.

We will be most interested in divisorial valuations. Recall that a divisorial valuation on $X$ is a valuation of the form $c\ord_E$, where $c\in \mathbb{Q}_{>0}$ and $E$ is a prime divisor over $X$. The set of divisorial valuations on $X$ is denoted by $X^{\Div}$. When $\mathbb{Q}_{>0}$ is replaced by $\mathbb{R}_{>0}$, we can similarly define a space $X^{\Div}_{\mathbb{R}}$.

Given any coherent ideal $\mathfrak{a}$ on $X$ and any $v\in X^{\An}$, we define 
\begin{equation}\label{eq:va}
v(\mathfrak{a}):=\min\{v(f):f\in \mathfrak{a}_{c(v)}\}\in [0,\infty],
\end{equation}
where $c(v)$ is the center of the valuation $v$ on $X$. 

Given any valuation $v$ on $X$, the Gauss extension of $v$ is a valuation $\sigma(v)$ on $X\times \mathbb{A}^1$:
\[
\sigma(v)\left(\sum_i f_i t^i\right):=\min_i (v(f_i)+i).
\]
Here $t$ is the standard coordinate on $\mathbb{A}^1=\Spec \mathbb{C}[t]$. The key property is that when $v$ is a divisorial valuation, then so it $\sigma(v)$. See \cite[Lemma~4.2]{BHJ17}.

\paragraph{Non-Archimedean plurisubharmonic functions.}
Let $X$ be an irreducible complex projective variety of dimension $n$ and $L$ be a holomorphic pseudoeffective $\mathbb{Q}$-line bundle on $X$. Through the GAGA morphism $X^{\An}\rightarrow X$ of ringed spaces, $L$ can be pulled-back to an analytic line bundle $L^{\An}$ on $X$. The purpose of this section is to study the psh metrics on $L^{\An}$. We will follow the approach of \cite{BJ18b}, which avoids the direct treatment of $L^{\An}$ itself.

Following \cite[Definition~2.18]{BJ18b}, we define $\mathcal{H}^{\gf}_{\mathbb{Q}}(L)$, the set of \emph{(rational)  generically finite Fubini--Study functions}  $\phi:X^{\An}\rightarrow [-\infty,\infty)$, that are of the following form:
    \begin{equation}
    \phi=\frac{1}{m}\max_j \{\log|s_j|+\lambda_j\}.
    \end{equation}
Here $m\in \mathbb{Z}_{>0}$ is an integer such that $L^m$ is a line bundle, the $s_j$'s are a finite collection of non-vanishing sections in $H^0(X,L^m)$, and $\lambda_j\in \mathbb{Q}$. We followed the convention of Boucksom--Jonsson by writing $\log |s_j|(v)=-v(s_j)$.

Now we come to the main definition of this paragraph:
\begin{definition}[{\cite[Definition~4.1]{BJ18b}}]\label{def:BJpshmetric}
A psh metric on $L^{\An}$ is a function $\phi:X^{\An}\rightarrow [-\infty,\infty)$ that  is not identically $-\infty$, and  is the pointwise limit of a decreasing net $(\phi_i)_{i\in I}$, where $\phi_i\in \mathcal{H}^{\gf}_{\mathbb{Q}}(L_i^{\An})$ for some $\mathbb{Q}$-line bundles $L_i$ on $X$ satisfying $c_1(L_i)\to c_1(L)$ in $\NS^1(X)_{\mathbb{R}}$.

\end{definition}

The set of psh metrics on $L^{\An}$ is denoted by $\PSH(L^{\An})$.
We endow $\PSH(L^{\An})$ with the topology of pointwise convergence on $X^{\Div}$. This topology is Hausdorff as functions in $\PSH(L^{\An})$ are completely determined by their restriction on $X^{\Div}$:
\begin{theorem}[{\cite[Theorem~4.22]{BJ18b}}]\label{thm:pshdetbydiv}
Let $\phi \in\PSH(L^{\An})$ and $\psi:X^{\An}\rightarrow [-\infty,\infty)$ be an usc function. Assume that $\phi\leq \psi$ on $X^{\Div}$, then the same holds on $X^{\An}$.
\end{theorem}

Next we note that we may use sequences instead of nets in the definition of $\PSH(L^{\An})$:
\begin{theorem}[{\cite[Corollary~12.18]{BJ18b}}]\label{thm:BJappseq}
Let $S$ be an ample line bundle on $X$. Let $\phi \in \PSH(L^{\An})$.
Then there is a sequence of rational numbers $\varepsilon_i\searrow 0$ and a decreasing sequence $\phi_i\in \mathcal{H}^{\gf}_{\mathbb{Q}}((L+\varepsilon_i S)^{\An})$ such that $\phi$ is the pointwise limit of $\phi_i$, as $i\to\infty$.
\end{theorem}

The space $\PSH(L^{\An})$ inherits most of the expected properties of (Archimedean) psh functions (\cite[Theorem~4.7]{BJ18b}). However, the following compactness result is not known:
\begin{conjecture}[{\cite[\S 5]{BJ18b}}]\label{conj:env}
Assume that $X$ is unibranch, then every bounded from above increasing net of elements in $\PSH(L^{\An})$ converges in $\PSH(L^{\An})$.
\end{conjecture}
This prediction is equivalent to so-called  envelope conjecture \cite[Conjecture~5.14]{BJ18b}: the regularized supremum of a bounded from above family of functions in $\PSH(L^{\An})$ lies in $\PSH(L^{\An})$. See \cite[Theorem~5.11]{BJ18b} for the proof of the equivalence. This conjecture is proved when $X$ is smooth and $L$ is nef in \cite{BJ18b}. More recently, in \cite{BJ22}, Boucksom--Jonsson further established the case when $X$ is smooth and $L$ is pseudoeffective.

\section{Non-Archimedean psh functions and the envelope conjecture}

In this section, we study the space of non-Archimedean psh functions. The main technical difficulty is to find the correct definition of these objects for a big/pseudoeffective cohomology class. As our class $\{\theta\}$ may be transcendental, we will give a completely analytic definition. However, we will point out that  our choices coincide with the analogous algebraic notions of \cite{BJ18b}, whenever $\{\theta\}$ is the first Chern class of a $\mathbb Q$-line bundle.

\subsection{Relative test curves and non-Archimedean metrics}\label{subsec:testcurvpresing}
Let $X$ be a connected projective manifold of dimension $n$. Let $\theta$ be a closed smooth real $(1,1)$-form on $X$ representing a big cohomology class $\{\theta\}$.

For $\varphi\in \PSH(X,\theta)$, we define the analytification $\varphi^{\An}:X^{\An}\rightarrow [-\infty,0]$ as follows:
\begin{equation}\label{eq:varphianv}
\varphi^{\An}(v):=-v(\varphi)=-\lim_{k\to\infty}\frac{1}{k}v\left(\mathcal{I}(k\varphi) \right)\,.
\end{equation}

The quantity inside the limit  is defined in \eqref{eq:va}. In addition, by the subadditivity of multiplier ideals, we have that $\mathcal{I}((k+k')\varphi)\subseteq \mathcal{I}(k\varphi)\mathcal{I}(k'\varphi).$
It follows that
\[
v(\mathcal{I}((k+k')\varphi))\geq v(\mathcal{I}(k\varphi))+v(\mathcal{I}(k'\varphi)).
\]
In particular, thanks to Fekete's lemma, the limit in \eqref{eq:varphianv} exists.

When $v=c\ord_E$ for some prime divisor $E$ over $X$, $\varphi^{\An}(v)=-c\nu(\varphi,E)$, by \cite[Lemma~B.4]{BBJ21}. Here $\nu(\varphi,E)$ is the Lelong number of $\varphi$ along $E$ (cf. \cite[(13)]{DZ22}).

The  analytification of a relative test curve $\{\psi_{\tau}\}_\tau$ is $\psi^{\An}:X^{\Div} \to [-\infty,0]$, defined as
\begin{equation}\label{eq:psiandef}
\psi^{\An}(v):=\sup_{\tau\leq \tau^+_\psi}(\psi_{\tau}^{\An}(v)+\tau).
\end{equation}

With the convention $(-\infty)^{\An}(v) = -\infty$, we can even allow $\tau \in \mathbb R$ in the supremum of \eqref{eq:psiandef}. Since $\tau \mapsto \psi_\tau$ is $\tau$-decreasing, we observe that it suffices to take supremum over $\tau<\tau_{\psi}^+$ in \eqref{eq:psiandef}. 

We point out that $\psi^{\An}$ can be computed using the subgeodesic corresponding to the test curve:

\begin{proposition}\label{prop:psiandPhi}
Let $\{\psi_\tau\}_\tau$ be relative a test curve with $\tau^+_\psi \leq 0$. Let $\Psi$ be the potential on $X\times \Delta^*$ corresponding to $\{\check \psi_t\}_t$ given by $\Psi(x,\xi):= {\check \psi}_{-\log|\xi|^2}(x)$ (recall \eqref{eq: Phi_def_subgeod}).  Since $\tau^+_\psi\leq 0$, $\Psi$ extends to a qpsh potential on $X\times \Delta$. Moreover,
\begin{equation}\label{eq:psiansigma}
\psi^{\An}(v)=-\sigma(v)(\Psi)\quad \text{for }v\in X^{\Div}.
\end{equation}
\end{proposition}

We will often refer to $\Psi$ as the \emph{potential of the test curve} $\{\psi_\tau\}_\tau$. The right-hand side of \eqref{eq:psiansigma} agrees with $\check{\psi}^{\An}(v)$ as defined in \cite{BBJ21}, when $\theta=c_1(L,h)$ for some Hermitian ample line bundle $(L,h)$ (see \cite[Section 4.3]{BBJ21}).
\begin{proof}
The proof is the same as \cite[Proposition~3.13]{DX20}, that deals with the ample case. We briefly recall the argument. We start with observing that
\[
\Psi(x,\delta)=\sup_{\tau\leq \tau^+_\psi} (\psi_{\tau}(x)-\log|\delta|^2\tau)\quad \text{for }x\in X, \delta\in \Delta^*.
\]
By definition of $\sigma(v) \in (X \times \Delta)^\Div$ we have that 
$\sigma(v)(\psi_\tau(x) - \log|\delta|^2 \tau) = v(\psi_\tau) - \tau.$
Lastly, since $\sigma(v)$ is a divisorial valuation on $X \times \Delta$, by \cite[Lemma~3.14]{DX20}, we conclude that 
$
\sigma(v)(\Psi)=\inf_{\tau\leq \tau^+_\psi} (v(\psi_\tau)-\tau)\,,$ finishing the proof.
\end{proof}

\paragraph{Piecewise linear curves.}
Next we introduce the notion of a \emph{piecewise linear curve} $\{\psi_{\tau}\}_\tau$ in $\PSH(X,\theta)$ associated with $\psi_{\tau_j} \in \textup{PSH}(X,\theta)$, for a finite number of parameters $\tau_0>\tau_1> \dots > \tau_N$. The piecewise linear curve is the affine interpolation of this data:\smallskip\\
\noindent $(i)$ $\psi_{\tau}=\psi_{\tau_N}$ for $\tau\leq \tau_N$.\smallskip\\
\noindent $(ii)$ For $t\in (0,1)$, we have $
    \psi_{(1-t)\tau_i+t\tau_{i+1}}=(1-t)\psi_{\tau_i}+t\psi_{\tau_{i+1}}\,.$\smallskip\\
\noindent $(iii)$  $\psi_{\tau}=-\infty$ for $\tau>\tau_0$. \smallskip

Observe that $\{\psi_{\tau}\}_\tau$ is $\tau$-usc but may not be $\tau$-concave. Despite this we introduce the analytification $\psi^\An$ of $\{\psi_\tau\}_\tau$ as follows:
\begin{equation}\label{eq:psianvlinear}
\psi^{\An}(v):=\sup_{\tau\leq \tau_0} (\psi_{\tau}^{\An}(v)+\tau)=\max_{\tau_i} (\psi_{\tau_i}^{\An}(v)+\tau_i)\quad \text{for }v\in X^{\An}.
\end{equation}

Given a bounded from above usc function $f$ defined on an interval of $\mathbb{R}$, the concave envelope $\tilde{f}$ of $f$ is the minimal concave function lying above $f$. Recall that $f$ and $\tilde{f}$ have the same inverse Legendre transform.
\begin{lemma}\label{lma:psieqpsiprimeondiv}
Let $\{\psi_{\tau}\}_\tau$ be a  piecewise linear curve in $\PSH(X,\theta)$, then the $\tau$-concave envelope $\{\tilde \psi_{\tau}\}_\tau$ of $\{\psi_{\tau}\}_\tau$ is a relative test curve. Moreover,
\begin{equation}\label{eq:psieqsiprime}
\psi^{\An}=\tilde{\psi}^{\An}\quad \text{on }X^{\Div}.
\end{equation}
\end{lemma}
\begin{proof}
For the first part, recall that $\{\tilde \psi_{\tau}\}_\tau$ is the Legendre transform of the inverse Legendre transform
\[
\check{\psi}_t:=\sup_{\tau\in \mathbb{R}} (\psi_{\tau}+t\tau)\quad \text{for }t> 0
\]
of $\{\psi_{\tau}\}_\tau$. So by Theorem~\ref{thm:RWNpresc}, it suffices to show that $\{\check{\psi}_t\}_t$ is a subgeodesic.  This is clear because
\[
\check{\psi}_t=\max_{j=0,\dots,N}(\psi_{\tau_j}+t\tau_j).
\]
Each term in the maximum is clearly a subgeodesic, hence so is $\{\check{\psi}_t\}_t$.

In order to prove \eqref{eq:psieqsiprime}, we may assume that $\psi_{\tau}=-\infty$ when $\tau>0$.
Note that by the same arguments as Proposition~\ref{prop:psiandPhi}, $\psi^{\An}(v)=-\sigma(v)(\Psi)$, where $\Psi$
is the potential on $X\times \Delta$ corresponding to the subgeodesic $\{\check{\psi}_t\}_t$.
So \eqref{eq:psieqsiprime} follows from Proposition~\ref{prop:psiandPhi} applied to $\{\tilde \psi_\tau\}_\tau$, and $\tau$-Legendre duality.
\end{proof}

\subsection{Transcendental non-Archimedean metrics} 
Let $X$ be a connected compact Kähler manifold of dimension $n$. Let $\theta$ be a closed smooth real $(1,1)$-form on $X$ representing a big cohomology class $\{\theta\}$.

We are ready to introduce the set of non-Archimedean psh metrics for a general pseudoeffective transcendental class $\{\theta\}$. Let $\mathcal{K}$ denote the set of K\"ahler metrics on $X$ endowed with the partial order: $\omega\succeq \omega'$ if $\omega\leq \omega'$ as forms. Clearly $\mathcal{K}$ is a directed set.

For $\omega\leq  \omega'$ there is a natural transition map from $\textup{MTC}_\mathcal I(\theta+\omega) \mapsto \textup{MTC}_\mathcal I(\theta+\omega')$ described as follows (recall \eqref{eq: MTC_def}). To $\{\psi_\tau\}_\tau \in \textup{MTC}_\mathcal I(\theta+\omega)$ we associate $\{\psi'_\tau\}_\tau \in \textup{MTC}_\mathcal I(\theta+\omega')$, where 
\begin{equation}\label{eq: transition_map}
    \psi'_\tau:=\left\{
        \begin{aligned}
            P_{\theta+\omega'}[\psi_{\tau}]_{\mathcal{I}},& \quad \text{if }\tau<\tau^+_{\psi};\\
            \lim_{\tau \nearrow \tau^+_\psi} \psi'_\tau, & \quad \text{if } \tau=\tau^+_{\psi};\\
            -\infty, & \quad \text{if } \tau>\tau^+_{\psi}.
        \end{aligned}
    \right.
\end{equation}
With this in hand, we can define the space of non-Archimedean functions:

\begin{definition}
\label{def:PSH-an} The space of  non-Archimedean functions of a  pseudoeffective class $\{\theta\}$
is the following projective limit in the category of sets:
\[
\PSH^{\NA}(\theta):=\varprojlim_{\omega\in \mathcal{K}}\textup{MTC}_\mathcal I(\theta+\omega).
\]
\end{definition}

Notice that the above limit is well-defined since $P_{\theta + \omega''}[P_{\theta + \omega'}[\psi_\tau]_\mathcal I]_\mathcal I = P_{\theta + \omega''}[\psi_\tau]_\mathcal I$ if $\omega \leq \omega' \leq \omega''$ and $\{\psi_\tau\}_\tau \in \textup{MTC}_\mathcal I(\theta+\omega)$.

We introduce a partial order on $\PSH^{\NA}(\theta)$: for $\phi=\{\{\phi^{\omega}_\tau\}_\tau\}_{\omega\in \mathcal{K}},\ \phi'=\{\{\phi'^{\omega}_\tau\}_\tau\}_{\omega\in \mathcal{K}}\in\PSH^{\NA}(\theta)$, we say $\phi\leq \phi'$ if $\phi^{\omega}_\tau\leq \phi'^{\omega}_\tau$ for all $\omega \in \mathcal{K}$ and $\tau \in \mathbb R$.

Observe that for $\phi=\{\{\phi^{\omega}_\tau\}_\tau\}_{\omega\in \mathcal{K}}\in \PSH^{\NA}(\theta)$,
$\tau^+_{\phi^{\omega}}$ does not depend on the choice of $\omega$. We denote the common value by $\tau^+_{\phi}$.

When $\{\theta\}$ is big, note that there exists a natural inclusion $\textup{MTC}_\mathcal I(\theta) \hookrightarrow \PSH^{\NA}(\theta)$. Indeed, with $\{\phi_\tau\}_\tau \in \textup{MTC}_\mathcal I(\theta)$ one simply associates $\{\{\phi^\omega_\tau\}_\tau\}_\omega \in \PSH^{\NA}(\theta)$, where
\begin{equation}\label{eq: MCTI_embed}
    \phi^\omega_\tau:=\left\{
        \begin{aligned}
            P_{\theta + \omega}[\phi_\tau]_\mathcal I,& \quad \text{if }\tau<\tau^+_{\phi};\\
            \lim_{\tau \nearrow \tau_{\phi}^+}\phi^{\omega}_{\tau}, & \quad \text{if } \tau=\tau^+_{\phi};\\
            -\infty, & \quad \text{if } \tau>\tau^+_{\phi}.
        \end{aligned}
    \right.
\end{equation}

Lastly, let us note that the counterpart of Conjecture \ref{conj:env} naturally holds in the setting of $\textup{PSH}^\NA(\theta)$.

\begin{theorem}\label{thm:increasingnettestcurveconv} 
Suppose that $\{\phi^i\}_{i \in I} \subset \textup{PSH}^\NA(\theta)$ is an increasing net with
\[
\sup_{i\in I}\tau^+_{\phi^i}<\infty.
\]
Then there exists $\phi \in  \textup{PSH}^\NA(\theta)$ such that for any $\omega\in \mathcal{K}$, $\phi^{i,\omega}_\tau \nearrow \phi^{\omega}_{\tau}$ almost everywhere for all $\tau<\tau^+_{\phi}$ and 
\[
\tau^+_{\phi}=\sup_{i\in I}\tau^+_{\phi^i}.
\]
\end{theorem}
\begin{proof}
Let $\omega \in \mathcal K$. By Lemma~\ref{lma:incfamily} there exists $\{\phi^\omega_\tau\}_\tau \in \textup{MCT}_\mathcal I(\theta)$ such that $\phi^{i,\omega}_\tau \nearrow \phi^{\omega}_\tau$ almost everywhere for all $\tau < \lim_i\tau^+_{\phi^i}=\tau^+_{\phi}$ and $\phi^{\omega}_{\tau^+_\phi} = \lim_{\tau \nearrow \tau^+_\phi} \phi^{\omega}_\tau$ almost everywhere. 

Let $\omega \leq \omega'$. Since $\phi^\omega_\tau \simeq_{\mathcal I} \phi^{\omega'}_\tau$ for all $\tau < \tau^+_\phi$ \cite[Proposition 2.18(ii)]{DX20}, we obtain that $\phi^{\omega'}_\tau = P_{\theta + \omega'}[\phi^\omega_\tau]_\mathcal I$, hence $\phi = \{\{\phi^\omega_\tau\}_\tau\}_\omega \in \textup{PSH}^\NA(\theta)$.
\end{proof}

\subsection{Comparison with Boucksom--Jonsson's NA metrics}
Let $X$ be a connected projective manifold of dimension $n$. Assume furthermore that $\{\theta\}=c_1(L)$ for some big $\mathbb{Q}$-line bundle $L$ on $X$. We will compare the sets $\textup{PSH}^\NA(\theta)$ introduced in the previous section and the space $\PSH(L^{\An})$ recalled in Section~\ref{subsec:NAformal}.

\begin{lemma}\label{lma:varphian}
For any $\varphi\in \PSH(X,\theta)$ we have that $\varphi^{\An}\in \PSH(L^{\An})$.
\end{lemma}
\begin{proof}

After replacing $L$ with a sufficiently high power, we may assume that $L$ is a line bundle. Take a very ample line bundle $H$ on $X$. 
By Siu's uniform global generation theorem \cite{Siu98}, \cite[Theorem~6.27]{Dem12} there exists $b>0$ large enough so that $H^{b}\otimes L^k\otimes \mathcal{I}(k\varphi)$ is globally generated  for all $k>0$. Let $\{s_i\}$ be a finite set of global sections that generate the sheaf $H^{b}\otimes L^k\otimes \mathcal{I}(k\varphi)$. Then $v(\mathcal{I}(k\varphi))=\min_i(v(s_i))$.
It follows that $v\mapsto -k^{-1}v\left(\mathcal{I}(k\varphi) \right)$ lies in $\mathcal{H}^{\gf}_{\mathbb{Q}}((L+\frac{b}{k}H)^{\An})$. Picking $k_m :=2 ^m$ and letting $m \to \infty$ we  conclude that $\varphi^{\An}\in \PSH(L^{\An})$.
\end{proof}

\begin{lemma}\label{lma:pltest}
Let $\{\psi_\tau\}_\tau$ be a piecewise linear curve in $\PSH(X,\theta)$.
Then $\psi^{\An}$, as defined in \eqref{eq:psianvlinear}, extends to an element $\psi^{\An}\in \PSH(L^{\An})$. 
\end{lemma}
\begin{proof}  The result follows from definition \eqref{eq:psianvlinear}, \cite[Theorem~4.7(ii)]{BJ18b} and Lemma~\ref{lma:varphian}.
\end{proof}

\begin{lemma}\label{lma:istab}
Let $R$ be a commutative $\mathbb{C}$-algebra of finite type and $I$ be an ideal of $ R[t]$. If for any $a\in S^1$, $a^*I\subseteq I$, then $I$ is stable under the $\mathbb C^*$-action. Moreover, there are ideals $I_0\subseteq I_1\subseteq \dots \subseteq I_m$ in $R$ so that
\begin{equation}\label{eq:Iexp}
    I=I_0 + I_1t+\dots + I_m (t^m),
    \end{equation}
\end{lemma}

\begin{proof}
It suffices to argue that $I$ can be expanded as in \eqref{eq:Iexp}.
To see this, assume that $a\in I$. We can write $a=a_0+a_1t+\dots+a_mt^m$ with $a_i\in R$. Then our assumption implies that $\sum_i a_i \rho^i t^i\in I$ as well for all $\rho\in S^1$. So by the Lagrange interpolation formula, $a_it^i\in I$ for all $i$. Therefore, we can write $I$ as $I_0+I_1t+I_2t^2+\dots$ for some ideals $I_0\subseteq I_1\subseteq \dots$ in $R$. But as $R$ is noetherian, there is $m\geq 0$ so that $I_{m'}=I_m$ for $m'>m$. \eqref{eq:Iexp} follows.
\end{proof}

\begin{lemma}\label{lma:GAGA}
Let $X$ be a complex projective variety and $p:X\times \mathbb{C}\rightarrow X$ be the natural projection. Assume that $\mathcal{I}$ is an analytic coherent ideal sheaf on $X\times \mathbb{C}$. Assume that $\mathcal{I}|_{X \times \mathbb C^*}=p^*\mathcal{J}$ for some coherent ideal sheaf $\mathcal{J}$ on $X$. Then $\mathcal{I}$ is the analytification of an algebraic coherent ideal sheaf.
\end{lemma}

\begin{proof}
    Let $q:X\times (\mathbb{P}^1\setminus \{0\})\rightarrow X$ be the natural projection. As $\mathbb C^* \subset \mathbb{P}^1\setminus \{0\}$ we can glue $q^*\mathcal{J}$ with $\mathcal{I}$ to get an analytic coherent ideal sheaf on $X\times \mathbb{P}^1$. By the GAGA principle, this ideal sheaf is necessarily algebraic, hence so is its restriction to $X\times \mathbb{C}$.
\end{proof}

Next we point out a version of Siu's uniform global generatedness lemma \cite{Siu88} that we will need in the proof of our next theorem: 

\begin{lemma}\label{lem: Siu_global_uniform generatedness} Let $L$ be a big line bundle on $X$ such that $c_1(L) = \{\theta\}$ and $\Phi \in \textup{PSH}(X\times \Delta, p_1^*\theta$), where $\Delta$ is the unit disk. Let $G$ be an ample line bundle on $X$. Then there exists $k>0$, only dependent on $X$ and $G$ such that $p_1^*(G^k \otimes L^m)\otimes \mathcal I(m \Phi)$ is globally generated for all $m \in \mathbb N$.   
\end{lemma}

\begin{proof} The argument for this is exactly the same as the one in \cite[Lemma~5.6]{BBJ21} with Nadal's vanishing replaced by the family version proved by Matsumura in  \cite[Theorem~1.7]{Mat16}.
 Alternatively, one can adapt the proof of \cite[Theorem 6.27]{Dem12} to our setting, with only one small change: instead of applying Nadel's theorem directly, one needs to use \cite[Corollary 5.3]{Dem12} when solving the $\bar\partial$-problem in the argument.
\end{proof}

\begin{proposition}\label{prop:PhiinducePSH} 
Let $\phi\in \PSH(X,\theta)$ be a model potential with positive mass.
Let $\Phi$ be the $p_1^*\theta$-psh function on $X\times \Delta$ corresponding to a psh geodesic ray $\{\phi_t\}_t$ in $\PSH(X,\theta)$ emanating from $\phi$, with $\sup \phi_1\leq 0$. Then the function
\[
v\mapsto -\sigma(v)(\Phi) \quad \text{for } v\in X^{\Div}
\]
admits a unique extension to an element in $\PSH(L^{\An})$.
\end{proposition}
\begin{proof}
We may assume that $L$ is a line bundle. Observe that the extension is unique if it exists by Theorem~\ref{thm:pshdetbydiv}.

\smallskip\noindent \textbf{Claim}. For each $m\in \mathbb{Z}_{>0}$, $\mathcal{I}(m\Phi)|_{X\times \Delta^*} = p_1^* \mathcal{I}(m\phi)|_{X\times \Delta^*} $. In particular,  $\mathcal{I}(m\Phi)$ admits an extension to a coherent ideal sheaf on $X\times \mathbb{C}$. \smallskip

To prove the claim it is enough to show that 
\begin{equation}\label{eq: sheaf_identity}
\mathcal{I}(m\Phi)|_{X \times \Delta^*} = \mathcal I((m \phi \circ p_1)|_{X \times \Delta^*})= p_1^* \mathcal{I}(m\phi)|_{X\times \Delta^*}. 
\end{equation}
We observe that $\mathcal I((m \phi \circ p_1)|_{X \times \Delta^*})\subseteq p_1^* \mathcal{I}(m\phi)|_{X\times \Delta^*}$ by \cite[Proposition~14.3]{Dem12} and the reverse inclusion is obvious. So it suffices to establish the first equality in \eqref{eq: sheaf_identity}.

If we could prove this, then the $\mathbb C^*$-invariant extension of $\mathcal{I}(m\Phi)$ to $X \times \mathbb C ^*$ would be simply $\mathcal I((m \phi \circ p_1)|_{X \times \mathbb C^*})$.

Consider the annulus $\Delta_{a,b} = \{z\in \mathbb{C}:\exp(-b)<|z|^2<\exp(-a)\}  \subset \Delta^*$ for $0<a<b$. To argue \eqref{eq: sheaf_identity}, it suffices to show that $\mathcal{I}(m\Phi)|_{X \times \Delta_{a,b}} = \mathcal I((m \phi \circ p_1)|_{X \times \Delta_{a,b}})$ for arbitrary $a,b$.

First we notice that due to convexity of (sub)geodesics in the time variable we have $\Phi(x,z)|_{X \times \Delta_{a,b}} \leq \phi(x) + C $ for some $C>0$.

Second, since $\phi$ has positive mass, due to  Lemma \ref{lem:Cproj} and \cite[Theorem~1.3]{DDNL2} we have that $\Phi(\cdot, \exp(-b)),\Phi(\cdot, \exp(-b)) \in \mathcal E(X,\theta,\phi)$. By \cite[Theorem~2.9]{Gu22}, there exists $\psi:= P_\theta(\Phi(\cdot, \exp(-a)),\Phi(\cdot, \exp(-b))) \in \mathcal E(X,\theta,\phi)$ such that $ \psi(z) \leq \Phi(x,z)|_{X \times \Delta_{a,b}}$. This last inequality follows from the comparison principle built into the definition of psh geodesic segments, recalled in \eqref{eq: psh_segm_comp}.
We conclude that
\begin{equation}\label{eq: containment_of_sheaves}\mathcal I((m \psi \circ p_1)|_{X \times \Delta_{a,b}}) \subseteq \mathcal{I}(m\Phi)|_{X \times \Delta_{a,b}} \subseteq \mathcal I((m \phi \circ p_1)|_{X \times \Delta_{a,b}}).
\end{equation}

By \cite[Theorem 1.3]{DDNL2} we have that $P(\psi + C,\phi) \nearrow P[\phi]=\phi$ a.e. on $X$ as $C \nearrow \infty$. Since $[P(\psi + C,\phi)] = [\psi]$, by Theorem \ref{thm:openness} we conclude that $\mathcal I((m \psi \circ p_1)|_{X \times \Delta_{a,b}}) = \mathcal I((m \phi \circ p_1)|_{X \times \Delta_{a,b}})$. Together with \eqref{eq: containment_of_sheaves} this finishes the proof of \eqref{eq: sheaf_identity}. \smallskip

From the claim and Lemma~\ref{lma:istab} and Lemma \ref{lma:GAGA} we get that
\begin{equation}\label{eq:ImPhiexp}
\mathcal{I}(m\Phi)=\mathfrak{a}_0 +\mathfrak{a}_1 t+\dots+\mathfrak{a}_{N-1} t^{N-1}+\mathfrak{a}_N (t^N)\,,
\end{equation}
where the $\mathfrak{a}_i$'s are coherent ideal sheaves on $X$. 

Using Lemma \ref{lem: Siu_global_uniform generatedness}, there exists $T \to X$ ample such that $p_1^* T\otimes L^m\otimes \mathcal{I}(m\Phi)$ is globally generated, which is equivalent to $T\otimes L^m\otimes \mathfrak{a}_i$ being globally generated for all $i$ (in contrast with the case where $\phi$ is bounded, explored in  \cite{BBJ21}, $\mathfrak{a}_N \neq \mathcal{O}_X$ in general).

We define
\[
\varphi_m(v):=-\frac{1}{m}\sigma(v)(\mathcal{I}(m\Phi))=-\frac{1}{m}\min_i (v(\mathfrak{a}_i)+i) ,\quad v\in X^{\Div}\,.
\]
From the right-hand side of the formula, $\varphi_m$ can be extended to an element in $\mathcal{H}^{\gf}_{\mathbb{Q}}((L+m^{-1}T)^{\An})$, which we denote by the same symbol.

For $v\in X^{\Div}$, it follows from the well-known argument using the Ohsawa--Takegoshi extension theorem (\cite[Lemma~B.4]{BBJ21}) that
\[
-\sigma(v)(\Phi)=\lim_{m\to\infty}-\frac{1}{2^m}\sigma(v)(\mathcal{I}(2^m\Phi))=\lim_{m\to\infty}\varphi_{2^m}(v)
\]
and the right-hand side defines an element in $\PSH(L^{\An})$ by definition, since $\{\varphi_{2^m}\}_m$ is decreasing.
\end{proof}

\begin{corollary}\label{cor:testcurvegivepshme}
Let $\{\psi_{\tau}\}_\tau$ be a test curve relative to $\phi \in \textup{PSH}(X,\theta)$. Then $\psi^{\An}: X^{\Div} \to \mathbb{R}$ admits a unique extension to   $\psi^{\An} \in \PSH(L^{\An})$. 
\end{corollary}

\begin{proof}
Observe that the extension is unique if it exists by Theorem~\ref{thm:pshdetbydiv}. We may assume that $\tau^+_\psi=0$, without loss of generality.

Let us first assume that $\phi$ has positive mass.
If $\{\chi_\tau\}_\tau$ is the maximization of $\{\psi_\tau\}_\tau$ then $\chi^\An = \psi^\An$ by definition. Hence, we can assume that $\{\psi_\tau\}_\tau$ is maximal, i.e., $\{\check{\psi}_t \}_t$ is a psh geodesic emanating from $P[\phi]$. The result now follows from Proposition~\ref{prop:PhiinducePSH} and Proposition~\ref{prop:psiandPhi}. 

In general, take an ample line bundle $S$ on $X$. Then the previous case shows that $\psi^{\An}\in \PSH((L+\epsilon S)^{\An})$ for any rational $\epsilon>0$. It follows that $\psi^{\An}\in \PSH(L^{\An})$ by \cite[Theorem~4.5]{BJ18b}.
\end{proof}

Before we can prove Theorem \ref{mainthm: isomorphism}, we deal with an intermediate case:

\begin{theorem}\label{thm:Imodelcurvenabij}
Assume that $\{\theta\}=c_1(L)$ for some big line bundle $L$ on $X$. 
The following hold:\smallskip\\
\noindent (i)  Let $\{\psi_\tau\}_\tau,\{\chi_\tau\}_\tau \in \textup{MTC}_\mathcal I(\theta)$. If $\psi^\An \geq \chi^\An$ then $\psi_\tau \geq \chi_\tau$ for any $\tau \in \mathbb R$. In particular, the map  $\{\psi_\tau\}_\tau \mapsto \psi^{\An}$ is an injection $\textup{MTC}_\mathcal I(\theta)\hookrightarrow \PSH(L^{\An})$.\smallskip\\
\noindent (ii) The image of the map $\textup{MTC}_\mathcal I(\theta) \ni \{\psi_\tau\}_\tau \mapsto \psi^{\An} \in \PSH(L^{\An})$ contains $\PSH((L-T)^{\An})$ for any ample $\mathbb{Q}$-line bundle $T$ on $X$ such that $L-T$ is big.
\end{theorem}

\begin{proof}
First observe that the map $\textup{MTC}_\mathcal I(\theta)\rightarrow \PSH(L^{\An})$ is well-defined by Corollary~\ref{cor:testcurvegivepshme}.

We argue  (i). Let $v\in X^{\Div}$ and $t \in \mathbb Q_+$. Then, using \eqref{eq:psiandef} we notice that
\begin{equation}\label{eq: legendre_NA}
t \psi^\An\Big (\frac{1}{t}  v \Big) =  \sup_{\tau\leq \tau^+_\psi}(\psi_{\tau}^{\An}(v)+ t \tau).
\end{equation}
 Using the condition $\psi^\An \geq \chi^\An$ we obtain that 
$$\sup_{\tau \in \mathbb R}(\psi_{\tau}^{\An}(v)+ t \tau) \geq \sup_{\tau\in \mathbb R}(\chi_{\tau}^{\An}(v)+ t \tau).$$
As a result, since $\mathbb Q_+$ is dense in $\mathbb R_+$, the above inequality holds for all $t \geq 0$.
Since $\tau \mapsto \psi_\tau^\An(v),\psi_\tau^\An(v)$ are both concave, taking the $t$-Legendre transform of both sides we conclude that $\psi^\An_\tau(v) \geq \chi^\An_\tau(v)$ for all $\tau \in \mathbb R$. Since the potentials  $\psi_\tau,\chi_\tau$ are $\mathcal I$-model, from \cite[Corollary~2.16]{DX20} we obtain that $\psi_\tau \geq \chi_\tau$ for all $\tau \in \mathbb R$. \smallskip

To argue (ii), 
we take $ \phi\in \PSH((L-T)^{\An}) \subset \PSH(L^{\An})$, and we want to write it as $\psi^{\An}$ for some $\psi\in \textup{MTC}_\mathcal I(\theta)$. Before we deal with this, let us only consider $\phi\in \mathcal{H}^{\gf}_{\mathbb{Q}}(L)$, say
\begin{equation}\label{eq:phiintominversemax}
\phi=m^{-1}\max_i (\log|s_i|+\lambda_i),    \end{equation}
where $s_i$ are a finite number of sections of $L^m$ and $\lambda_i\in \mathbb{Q}$.

We may assume that $\lambda_1\leq \lambda_2\leq \dots \leq \lambda_N$. Write $I_{\lambda}$ for the set of $i$ such that $\lambda_i=\lambda$. We denote the finitely many $\lambda$ so that $I_{\lambda}$ is non-empty as $\tau_0> \dots > \tau_N$.
Define a curve $\psi_{\tau}$ as follows:
\[
\psi_{\tau_i}=\frac{1}{m}\max_{i\in I_{\tau_i}} (\log |s_i|^2_{h^m}+\tau_i)\,.
\]
We define $\{\psi_{\tau}\}_\tau$ to be the piecewise linear curve associated with the $\psi_{\tau_i}$ (recall the definition preceding \eqref{eq:psianvlinear}). Let $\{\psi'_{\tau}\}_\tau$ be the $\tau$-concave envelope of $\{\psi_{\tau}\}_\tau$. 

By Lemma~\ref{lma:psieqpsiprimeondiv}, $\psi'^{\An}=\psi^\An=\phi$ on $X^{\Div}$. By Lemma~\ref{lma:pltest} and Theorem~\ref{thm:pshdetbydiv}, the same holds on $X^{\An}$. We can replace $\psi'_{\tau}$ with $P[\psi'_{\tau}]_{\mathcal{I}}$ when $\tau<\tau^+_\psi$ and $\psi'_{\tau^+}$ with the limit value of $P[\psi'_{\tau}]_{\mathcal{I}}$ as $\tau$ increases to $\tau^+_\psi$. Defined this way, $\{\psi'_\tau\}_\tau$ is an $\mathcal I$-maximal test curve, and we still have $\psi'^{\An}=\phi$ on $X^{\Div}$. However, we may not have that $\{\psi'_\tau\}_\tau \in \textup{MTC}_\mathcal I(\theta)$ as $\psi'_{-\infty}$ may not have positive mass.

Now we consider $\phi\in \mathcal{H}_{\mathbb{Q}}^{\gf}(L-T) \subset \mathcal{H}_{\mathbb{Q}}^{\gf}(L)$ for some ample $\mathbb{Q}$-line bundle $T$ such that $L-T$ is still big. We may assume that $T$ is a line bundle.
Fix a smooth strictly psh Hermitian metric $h_T$ on $T$ with Chern form $\omega$.

Let $\{\psi'_\tau\}_\tau \subset \textup{PSH}(X,\theta-\omega)$ be the $\mathcal I$-maximal test curve with respect to $\theta-\omega$, such that $\psi'^{\An} = \phi$, constructed above. We define $\chi_\tau \in \textup{PSH}(X,\theta)$ in the following manner:
$$\chi_\tau := P_{\theta}[\psi'_\tau]_\mathcal I, \ \ \tau < \tau^+_{\psi'},   \ \ \ \ \chi_{\tau^+_{\psi'}}:= \lim _{\tau \nearrow \tau^+_{\psi'}} \chi_\tau.$$
We get that $\{\chi_\tau\}_\tau \subset \textup{PSH}(X,\theta)$ is  $\mathcal I$-maximal and $\chi^\An = \phi$.

Finally, we only need to argue that 
$\{\chi_\tau\}_\tau\in \textup{MTC}_\mathcal I(\theta)$. This follows from the fact that $P_{\theta}[\psi'_\tau]_\mathcal I \geq P_{\theta}[\psi'_\tau]$ and \cite[Theorem 1.2]{WN19}:
\begin{equation}\label{calccc}
\int_X \theta_{\chi_\tau}^n = \int_X \theta_{P_{\theta}[\psi'_\tau]_\mathcal I}^n \geq\int_X \theta_{P_{\theta}[\psi'_\tau]}^n = \int_X ((\theta-\omega)_{\psi'_\tau} + \omega)^n \geq \int_X \omega^n.
\end{equation}
Finally, we deal with the case when $\phi\in \PSH((L-S)^{\An})$ for some ample $\mathbb{Q}$-line bundle $S$ on $X$.
By Theorem~\ref{thm:BJappseq}, we can take an ample line bundle $S$ on $X$, a decreasing sequence of rational numbers $c_i\to 0$ and a decreasing sequence $\phi_i\in \mathcal{H}^{\gf}_{\mathbb{Q}}((L-S+c_i S)^{\An}) \subset \mathcal{H}^{\gf}_{\mathbb{Q}}(L^{\An})$ converging to $\phi$. We will assume that $c_i\leq 1/2$ for all $i$.

Fix a smooth psh metric $h_S$ on $S$ with $\omega=c_1(S,h_S)>0$. By the previous step, we can find $\{\psi^i_\tau \}_\tau\in \textup{MTC}_\mathcal I(\theta)$ such that $\psi^{i,\An}=\phi_i$. Moreover, due to \eqref{calccc} we have that
\begin{equation}\label{eq: volume_Psi_i_est}
\int_X \theta_{\psi^i_\tau}^n  \geq (1 - c_i)^n \int_X \omega^n \geq \frac{1}{2^n}\int_X \omega^n , \ \ \tau \in \mathbb R.
\end{equation}

Since $(\phi^i)^\An \geq (\phi^{i+1})^\An$, due to Theorem~\ref{thm:Imodelcurvenabij}(i), we obtain that $\{\psi^i_\tau\}_\tau$ is $i$-decreasing. Let $\psi_{\tau}=\inf_i\psi^i_{\tau}$. By \cite[Proposition 4.6]{DDNL5} and Lemma~\ref{lma:decnetna}, $\{\psi_\tau\}_\tau\in \textup{MTC}_\mathcal I(\theta)$ with $\int_X \theta^n_{\psi_\tau} \geq \frac{1}{2^n} \int_X \omega^n$.

We need to show that $\psi^{\An}=\phi$. Using \eqref{eq: volume_Psi_i_est}, again, Lemma~\ref{lma:Imodeldeclelong} implies that $\psi_{\tau}^{\An}=\inf_i \psi_{\tau}^{i,\An}$ when $\tau<\tau^+_\psi$. To finish, we need to show that for any $v\in X^{\Div}$, 
\begin{equation}\label{eq:supinfeqinfsup}
\adjustlimits\sup_{\tau \in \mathbb R }\inf_i(\psi_{\tau}^{i,\An}(v)+\tau)=\adjustlimits\inf_i \sup_{\tau \in \mathbb R}(\psi_{\tau}^{i,\An}(v)+\tau).
\end{equation}
Due to Lemma \ref{lma:Imodeldeclelong}, we have that $\tau \mapsto \psi_{\tau}^{i,\An}(v)$ is concave and usc on $\mathbb R$. So \eqref{eq:supinfeqinfsup} follows from Lemma~\ref{lma:DC}, proved below.
\end{proof}

We state the following result from convex analysis, a special case of \cite[Theorem~2]{PZ04}:

\begin{lemma}\label{lma:DC}
Let $f_i:\mathbb{R}\rightarrow [-\infty,\infty)$ ($i\in I$) be a monotone net of proper usc concave functions and $f:\mathbb{R}\rightarrow [-\infty,\infty)$ another proper usc concave function. Assume that $f_i(\tau)\to f(\tau)$ for all $\tau \in \mathbb{R} $.
Then
\[
\check{f}=\lim_i \check{f}_i.
\]
\end{lemma}

Recall the definition of $\PSH^\NA(\theta)$ from Definition \ref{def:PSH-an}. When $\{\theta\} = c_1(L)$ for some pseudoeffective $\mathbb Q$-line bundle $L$ we now define the map:
\begin{equation}\label{eq: BJ_map_def}
\textup{PSH}^{\NA}(\theta) \ni \phi \mapsto \phi^\An \in \textup{PSH}(L^\An).
\end{equation}

Let $\phi = \{\{\phi^\omega_\tau\}_\tau \}_{\omega \in \mathcal K} \in \textup{PSH}^{\NA}(\theta)$.
Let $\omega \in \mathcal K$ be such that $\{\omega\} = c_1(T)$, for a $\mathbb Q$-ample line bundle $T$. We get that $\{\phi^\omega_\tau\}_\tau \in \textup{MTC}_\mathcal I(\theta + \omega)$, hence $(\phi^\omega)^\An \in \textup{PSH}((L+ T)^\An)$ by Theorem \ref{thm:Imodelcurvenabij}(i). We make the following preliminary definition:
\begin{equation}\label{eq: an_candidate_def}
\phi^\An := (\phi^\omega)^\An.
\end{equation}
Among other things, we need to show that this definition is independent of the choice of $\omega.$ Let $\omega' \in \mathcal K'$ such that $\{\omega'\} = c_1(T')$, for some $\mathbb Q$-ample line bundle $T'$ and $\omega \leq \omega'$. Then we have that
$$P_{\theta + \omega'}[\phi^\omega_\tau]_\mathcal I = \phi^{\omega'}_\tau, \ \ \ \  \tau < \tau^+_\phi.$$

As a result, using \cite[Proposition~2.18(ii)]{DX20}, we get that $(\phi^\omega_\tau)^\An = (\phi^{\omega'}_\tau)^\An, \ \tau < \tau^+_\phi$. Comparing with \eqref{eq:psiandef}, we arrive at $(\phi^\omega)^\An = (\phi^{\omega'})^\An$, as desired. 

Moreover, our above analysis also shows that $\phi^\An \in \bigcap_{L'}\PSH((L+L')^{\An})$, where $L'$ runs over all ample $\mathbb{Q}$-line bundles on $X$. The latter space is equal to $\PSH(L^{\An})$ by \cite[Theorem~4.5]{BJ18b}, so our map $\phi \mapsto \phi^\An$ from \eqref{eq: an_candidate_def} is indeed well-defined.

Observe that 
\begin{equation}\label{eq:trivialvaluationvaluetauplus}
    \phi^{\An}(v_{\mathrm{triv}})=\tau^+_{\phi},
\end{equation}
where $v_{\mathrm{triv}}$ denotes the trivial valuation.

We show that the map of \eqref{eq: BJ_map_def} is actually a bijection, giving a transcendental interpretation of the non-Archimedean metrics of Boucksom--Jonsson:

\begin{theorem}\label{thm:NAtoNAbij}
Assume that $\{\theta\}=c_1(L)$ for some pseudoeffective $\mathbb{Q}$-line bundle $L$.
The map of \eqref{eq: BJ_map_def}  is a bijection. In addition, let $\phi,\chi \in \PSH^{\NA}(\theta)$. If $\phi^\An \leq \chi^\An$ then $\phi \leq \chi$.
\end{theorem}

\begin{proof} 
By scaling, we may assume that $L$ is a line bundle.
We first address the last statement. Let $\phi = \{\{\phi^\omega_\tau\}_\tau\}_\omega$ and $\chi = \{\{\chi^\omega_\tau\}_\tau\}_\omega$. Given that $\phi^\An \leq \chi^\An$, Theorem~\ref{thm:Imodelcurvenabij}(i) gives that $\phi^\omega_\tau \leq \chi^\omega_\tau, \ \tau \in \mathbb R, \omega\in \mathcal K$, implying that $\phi \leq \chi$. This immediately gives that $\phi \mapsto \phi^\An$
 is injective.

To address surjectivity, let $\chi \in \textup{PSH}(L^\An)$. Let $\omega \in \mathcal K$ such that $\{\omega\}=c_1(T)$ for a $\mathbb{Q}$-ample line bundle $T$. Then $\chi \in \textup{PSH}((L + T)^\An)$, and by Theorem \ref{thm:Imodelcurvenabij}(ii) there exists $\{\phi^\omega_\tau\} \tau \in \textup{MTC}_\mathcal I(\theta + \omega)$ such that $(\phi^\omega)^\An = \chi$.

Let $\omega' \in \mathcal K$ such that $\{\omega'\} = c_1(T')$, for some $\mathbb Q$-ample line bundle $T'$ and $\omega \leq \omega'$. Due to injectivity of $\textup{MTC}_\mathcal I(\theta + \omega')\mapsto \PSH((L + T')^{\An})$ we get that $P_{\theta + \omega'}[\phi^\omega_\tau]_\mathcal I = \phi^{\omega'}_\tau, \ \tau < \tau^+_\phi$.

For a non-rational form $\omega'' \in \mathcal K$, let $\omega \in \mathcal K$ such that $\{\omega\}=c_1(T)$ for a $\mathbb Q$-ample line bundle $T$ and $\omega \leq \omega''$. Then we define 
$$\phi^{\omega''}_\tau = P_{\theta + \omega''}[\phi^\omega_\tau]_\mathcal I, \ \ \ \tau < \tau^+_\phi.$$

It is immediate to see that the above definition is independent of the choice of $\omega$. We see that $\{\{\phi^\omega_\tau\}_\tau\}_\omega \in \textup{PSH}^{\NA}(\theta)$, and  $\phi^\An = \chi$, proving surjectivity. 
\end{proof}
\begin{remark}\label{rmk:realclass_iso}
Boucksom--Jonsson's theory extends to any pseudoeffective $(1,1)$-class $\{\theta\}$ in the Néron--Severi group of $X$ with real coefficients, giving a space of non-Archimedean metrics $\PSH(\{\theta\}^{\NA})$.
In this case we still have a canonical identification
\begin{equation}\label{eq:PSHNA_realclass_identi}
\PSH^{\NA}(\theta)\xrightarrow{\sim} \PSH(\{\theta\}^{\NA}).
\end{equation}
To see this, let $\mathcal{K}'$ be the directed set of $(1,1)$-classes $\{\theta'\}$ in the Néron--Severi group of $X$ with rational coefficients such that $\{\theta'\}-\{\theta\}$ is a K\"ahler class. Then we have natural identifications
\[
\PSH^{\NA}(\theta)\xrightarrow{\sim}\varprojlim_{\theta'\in \mathcal{K}'} \PSH^{\NA}(\theta'),\quad \PSH(\{\theta\}^{\NA})\xrightarrow{\sim}\varprojlim_{\theta'\in \mathcal{K}'} \PSH(\{\theta'\}^{\NA}).
\]
The former follows immediately from our definition, the latter follows from \cite[Section~4.1, (PSH2)]{BJ18b}. In particular,  \eqref{eq:PSHNA_realclass_identi} follows from Theorem~\ref{thm:NAtoNAbij}.
\end{remark}

Using our analysis we conclude that the envelope conjecture holds in our setting, as confirmed by Boucksom--Jonsson using non-Archimedean methods \cite{BJ22}:

\begin{corollary}\label{cor:env_conj}
    Conjecture~\ref{conj:env} holds for any pseudo-effective $\mathbb Q$-line bundle $L$ on $X$.
\end{corollary}

\begin{proof}
Take a smooth closed real $(1,1)$-form $\theta$ in $c_1(L)$. Let $\{\varphi^i\}_{i\in I}$ be a bounded from increasing net in $\PSH^{\NA}(L)$. By Theorem~\ref{thm:NAtoNAbij},  $\varphi^i=\phi^{i,\An}$ with $\phi^{i}\in \PSH^{\NA}(\theta)$. By the same theorem, $\{\varphi^i\}_{i\in I}$ is also an increasing net. Moreover, by \eqref{eq:trivialvaluationvaluetauplus},
\[
\sup_{i\in I}\tau^+_{\varphi^i}<\infty. 
\]
By  Theorem~\ref{thm:increasingnettestcurveconv}, we can find $\phi\in \PSH^{\NA}(\theta)$ such that for any $\omega \in \mathcal{K}$, $\phi^{i,\omega}_{\tau}\nearrow \phi^{\omega}_{\tau}$ almost everywhere for any $\tau<\tau^+_{\phi}$.

We claim that $\phi^{i,\An}$ converges to $\phi^{\An}$. 

For any fixed $\omega \in \mathcal K$, the usc property of the Lelong number \cite[Lemma~2.6]{Bo17} gives
\[
(\phi^\omega_{\tau})^{\An}(v)=\lim_i (\phi^{i,\omega}_{\tau})^{\An}(v)\quad \text{for }v\in X^{\Div},\tau<\tau_\psi^+.
\]
It follows that
\[
\phi^{\An}(v)=(\phi^\omega)^{\An}(v)=\sup_{\tau\in \mathbb{R}} ((\phi^{\omega}_\tau)^\An(v)+\tau)=\sup_{\tau\in \mathbb{R}}\sup_{i\in I} ((\phi^{i,\omega}_\tau)^\An(v)+\tau)=\sup_{i\in I}\phi^{i,\An}(v)
\]
for all $v\in X^{\Div}$, finishing the proof.
\end{proof}

\subsection{The non-Archimedean finite energy space}\label{subsec:nae1}
Let $X$ be a compact K\"ahler manifold of dimension $n$ and $\{\theta\}$ be a big cohomology class on $X$.
Let $V=\int_X \theta_{V_{\theta}}^n$ denote the volume of $\{\theta\}$.

Recall that $\textup{MTC}_\mathcal I(\theta)$ naturally embeds into $\PSH^{\NA}(\theta)$ \eqref{eq: MCTI_embed}. For $\{\eta_\tau\}_\tau\in \textup{MTC}_\mathcal I(\theta)$ we define the Monge--Amp\`ere energy $I^{\NA}_{\theta}:\textup{MTC}_\mathcal I(\theta)\rightarrow[-\infty,\infty)$ as
\begin{equation}\label{eq: INAdef}
I_{\theta}^{\NA}(\eta):=I_\theta \{\eta_\tau\} = \tau^+_\eta+\frac{1}{V}\int_{-\infty}^{\tau^+_\eta}\left(\int_X\theta_{\eta_{\tau}}^n-V\right)\,\mathrm{d}\tau.
\end{equation}
We extend $I^{\NA}_{\theta}$ to a function on $\PSH^{\NA}(\theta)$ by setting it to be $-\infty$ on $\PSH^{\NA}(\theta) \setminus \textup{MTC}_\mathcal I(\theta)$.

We then define  the non-Archimedean finite energy space as 
\[
\mathcal{E}^{1,\NA}(X,\theta):=\left\{\phi\in \PSH^\NA(\theta):I_{\theta}^{\NA}(\phi)>-\infty\right\}\,.
\]

Since $V \geq \int_X\theta_{\eta_\tau}^n$, for \eqref{eq: INAdef} to be finite $\int_X\theta_{\eta_\tau}^n$ must converge to $V$ as $\tau\searrow-\infty$. By \cite[Theorem 2.3]{DDNL2} this implies that $\int_X\theta_{\eta_{-\infty}}^n = V$. Since $\eta_{-\infty}$ is maximal, we obtain that $\eta_{-\infty} = V_\theta$, i.e., $\{\eta_\tau\}_\tau$ is a test curve relative to $V_\theta$, whenever $\{\eta_\tau\}_\tau \in \mathcal{E}^{1,\NA}(X,\theta)$.

Similarly, the non-Archimedean $\mathcal J$ functional is introduced as
\begin{equation}\label{eq:Jreformulate}
\mathcal J_\theta^\NA(\eta) = \tau^+_\eta - I_\theta^\NA(\eta) =\frac{1}{V}\int_{-\infty}^{\tau^+_\eta}\left(V-\int_X\theta_{\eta_{\tau}}^n\right)\,\mathrm{d}\tau.
\end{equation}
When $\{\theta\}=c_1(L)$ for some ample $\mathbb{Q}$-line bundle $L$, it has been pointed out in \cite[Theorem 1.1, Theorem 1.2]{DX20} that our definitions of $\mathcal{E}^{1,\NA}(X,\theta)$ and  $\mathcal J_\theta^\NA, I_\theta^\NA$ coincide with the ones given in \cite{BJ18b}. 

In the literature, the space $\mathcal{E}^1(L^{\An})$ has not yet been defined for a general big $\mathbb{Q}$-line bundle $L$. Above we gave an analytic definition, but it is desirable to have a purely non-Archimedean/algebraic definition as well.

After defining finite energy non-Archimedean spaces, we can finally relate the $\mathcal{I}$-maximal finite energy rays to non-Archimedean potentials, as a direct consequence of Theorem~\ref{thm:RWNpresc}:
\begin{theorem}
There is a bijective function $\{\psi_\tau\}_\tau \mapsto \{\check \psi_t\}_t$, between   $\mathcal{R}^1_{\mathcal{I}}(X,\theta)$ and $\mathcal{E}^{1,\NA}(X,\theta)$. Moreover, under this correspondence, the radial Monge--Amp\`ere energy corresponds to the non-Archimedean Monge--Amp\`ere energy defined in \eqref{eq: INAdef}.
\end{theorem}

\section{Approximations of \texorpdfstring{$\mathcal{I}$}{I}-maximal rays by filtrations and flag configurations}

In this section $X$ is a connected projective manifold of dimension $n$.
The purpose of this section is to show that an $\mathcal{I}$-maximal finite energy ray can be approximated by simpler rays induced by flag configurations or filtrations defined below.

\subsection{Filtrations of big line bundles}
For this subsection let $L$ be a big line bundle on $X$ and $h$ be a smooth Hermitian metric.
Let $\theta:=-\ddc\log h$. 
We set $\vol (L):=\int_X \theta_{V_{\theta}}^n$.
Denote the section ring of $L$ by $R(X,L)$:
\[
R(X,L):=\bigoplus_{m\in \mathbb{N}} R_m,\quad R_m:= H^0(X,L^{m}).
\]

\begin{definition}\label{def:filtration}
By a (bounded left-continuous multiplicative decreasing) filtration $\cF$ of the section ring $R(X,L)$, 
we mean a family of linear subspaces $\{\cF^\lambda R_m\}_{\lambda\in\RR}$ of $R_m$ for each $m\in\NN$ such that \vspace{0.1cm}\\
\noindent (i)$\cF^{\lambda}R_m\subseteq\cF^{\lambda^\prime}R_m$ whenever $\lambda\geq\lambda^\prime$;\vspace{0.1cm}\\
\noindent (ii) $\cF^\lambda R_m=\bigcap_{\lambda^\prime<\lambda}\cF^{\lambda^\prime} R_m$;\vspace{0.1cm}\\
\noindent (iii)$\cF^{\lambda_1}R_{m_1}\cdot\cF^{\lambda_2}R_{m_2}\subseteq\cF^{\lambda_1+\lambda_2}R_{m_1+m_2}$ for $\lambda_1,\lambda_2\in\RR$ and $m_1,m_2\in\NN$;\vspace{0.1cm}\\
\noindent (iv) there exists $C>0$ such that $\cF^{-Cm}R_m=R_m$ and $\cF^{Cm}R_m=\{0\}$ for all $m\in\NN$.

\end{definition}

Given a filtration $\cF$ of $R(X,L)$ and $m\in \mathbb{N}$, set
\[
\tau_m(\cF):=\max \{\lambda\in\RR: \cF^\lambda R_m\neq\{0\} \}.
\]
Then clearly $\tau_m(\cF)+\tau_{k}(\cF)\leq\tau_{m+k}(\cF)$ for any $m,k\in\NN$, so by Fekete's lemma one can put
$$
\tau_{L}(\cF):=\lim_{m\rightarrow\infty}\frac{\tau_m(\cF)}{m}=\sup_{m\in\NN}\frac{\tau_m(\cF)}{m}.
$$
Moreover, for any $\tau<\tau(\cF)$, set
$$
\vol(\cF R^{(\tau)}):=\varlimsup_{m\rightarrow\infty}\frac{\dim\cF^{\tau m}R_m}{m^n/n!}
$$
and
$$
S_{L}(\cF):=\tau_{L}(\cF)+\frac{1}{\vol(L)}\int_{-\infty}^{\tau_{L}(\cF)}\left(\vol(\cF R^{(\tau)})-\vol(L)\right)\,\mathrm d\tau.
$$
It is trivial to see that $S_{L}(\mathcal F) \leq \tau_{L}(\mathcal F)$ in general.

Let $E$ be any prime divisor over $X$. It induces a natural filtration $\cF_E$ on the section ring $R(X,L)$. More precisely, for $\lambda\in\RR$ and $m\geq 1$ one puts
\[
\cF_{E}^\lambda R_m:=\{s\in R_m=H^0(X,L^{m}):\ord_E(s)\geq \lambda\}.
\]

In this case, we set for simplicity
\[
\tau_L(E):=\tau_{L}(\cF_{E})\text{ and }S_L(E):=S_{L}(\cF_{E}).
\]

\begin{lemma}
\label{lem:S-E<tau-E}
    For any prime divisor $E$ over $X$, one has $S_L(E)<\tau_L(E)$.
\end{lemma}

\begin{proof}
    Consider $f(\tau):=\vol(\cF_E R^{(\tau)})=\vol(L-\tau E)$ for $\tau\in[0,\tau_L(E)]$. This is a continuous and decreasing function with $f(0)=\vol(L)>0$ and $f(\tau_L(E))=0$. So one must have
    $$
S_L(E)=\frac{1}{\vol(L)}\int_0^{\tau_L(E)}f(\tau)\,\mathrm{d}\tau<\frac{1}{\vol(L)}\int_0^{\tau_L(E)}f(0)\,\mathrm{d}\tau=\tau_L(E).
    $$
\end{proof}

\begin{definition}\label{def:testcdefbyfil}
  Given any filtration $\cF$ of $R(X,L)$, for $\tau<\tau_{L}(\mathcal F)$ and $m\in \mathbb{Z}_{>0}$, let
\[
\phi^{\cF}_{\tau,m}:=\usc\sup\left\{\frac{1}{m}\log|s|^2_{h^{m}}: s\in\cF^{\tau m}H^0(X,L^{m}), \ \ |s|^2_{h^{m}}\leq1\right\}
 \textup{ and }
\phi^{\cF}_\tau:=\usc\sup_m\phi^{\cF}_{\tau,m}.
\]
Observe that $m\phi^{\cF}_{\tau,m}$ is super-additive, so by Fekete's lemma,
\[
\phi^{\cF}_\tau=\usc \lim_{m\to\infty}\phi^{\cF}_{\tau,m}.
\]  
Setting $\phi^\mathcal F_{\tau(\mathcal F)} := \lim_{{\tau \nearrow \tau(\mathcal F)}} \phi_\tau$ and $\phi^\mathcal F_\tau = -\infty$ if $\tau > \tau(\mathcal F)$, we get that $\phi^{\mathcal{F}}=\{\phi^{\mathcal{F}}_{\tau}\}_{\tau}$ is a bounded test curve, called the \emph{test curve associated with} $\mathcal{F}$.
\end{definition}

When working with a divisor $E$ over $X$ we will use the notation 
\[
\phi^{E}_{\tau,m}:=\phi^{\cF_{E}}_{\tau,m}, \quad \phi^{E}_\tau:=\phi^{\cF_{E}}_\tau.
\]

But we need to point out that, as opposed to the ample setting studied in \cite{RWN14,DX20}, it is not clear if the test curves thus constructed are \emph{maximal}  in the big case. We conjecture this to be the case, but as we shall see, this uncertainty will not cause any issue for the discussions below.

\begin{definition}[{\cite{PS07,RWN14}}]\label{def:PSray1}
We define the \emph{Phong--Sturm ray} $\{r^\mathcal F_t\}_t $ associated with a filtration $\mathcal F$ of $R(X,L)$
to be the inverse Legendre transform of the maximization 
of $\{\phi^\mathcal F_\tau\}_\tau$, defined in \eqref{eq: maximization_def}. 
\end{definition}

Due to boundedness of $\mathcal F$, we immediately notice that $r^\mathcal F_0= V_\theta$ and $\hat r_\tau^\mathcal F = V_\theta$ for $\tau \leq - C$. This implies that the  potentials $\hat r^\mathcal F_\tau, \phi^\mathcal F_\tau, \ \tau < \tau(\mathcal F)$ have non-zero mass. 

According to the next lemma, the Phong-Sturm ray is $\mathcal I$-maximal:

\begin{lemma}\label{lem: max_I_max_filt} We have that $\hat r^\mathcal F_\tau = P[\phi^\mathcal F_\tau] = P[\phi^\mathcal F_\tau]_\mathcal I$ for any $\tau < \tau_L(\mathcal F)$.
\end{lemma}

\begin{proof} 
By Fekete's lemma, $\psi_m: = \phi^ \mathcal F_{\tau, 2^m} \nearrow \phi^\mathcal F_\tau \leq P[\phi^\mathcal F_\tau].$ This implies that $P[\psi_m]_\mathcal I = P[\psi_m] \leq P[\phi^\mathcal F_\tau]$, by \cite[Proposition~2.20]{DX20}.

Let $\chi:= \lim_m P[\psi_m]_\mathcal I$. Since $\tau  < \tau_L(\mathcal F)$ we have that $\chi = P[\chi]_\mathcal I$, by \cite[Lemma~2.21(iii)]{DX20}. Since $\phi^\mathcal F_\tau \leq \chi \leq  P[\phi^\mathcal F_\tau]$, and $\mathcal I$-model potentials are model \cite[Proposition~2.18]{DX20} we obtain that $\chi = P[\phi^\mathcal F_\tau] $, hence $P[\phi^\mathcal F_\tau] = P[\phi^\mathcal F_\tau]_\mathcal I$.
\end{proof}

Lastly, we note the following formula for the radial Monge--Amp\`ere energy of the Phong--Sturm ray:

\begin{proposition}\label{prop:I=S}
    Let $\cF$ be a filtration of $R(X,L)$. Then $I_\theta\{ r^\cF_t\}=S_L(\cF)$.
\end{proposition}

\begin{proof}

    First, by \cite[(25)]{DZ22},
    \[
    I_\theta\{ r^\cF_t\}=I_\theta\{\check \phi^\cF_t\}=\tau_L(\cF)+\frac{1}{\vol(L)}\int_{-\infty}^{\tau_L(\cF)}\left(\int_X\theta^n_{\phi^\cF_\tau}-\vol(L)\right)\mathrm d\tau.
    \]
    So it is enough to argue that 
    $\int_X\theta^n_{\phi^\cF_\tau}=\vol(\cF R^{(\tau)})$
    for any $\tau<\tau_L(\cF)$. 
    This is a consequence of \cite[Theorem~1.3]{His13}. Indeed, consider the graded linear series 
    $\cF R^{(\tau)}:=\{\cF^{\tau m}R_m\}_{m\in\NN}$. It contains an ample linear series in the sense of \cite[Lemma 1.6]{BC11}. Then for all $m$ sufficiently divisible the natural map $X\dashrightarrow\mathbb P(\cF^{\tau m}R_m)^*$ is birational to its image. So \cite[Theorem 1.3]{His13} implies that $\int_X\theta^n_{\phi^\cF_\tau}=\vol(\cF R^{(\tau)})$, as wished.
\end{proof}

\subsection{Flag configurations}

\begin{definition}\label{def:flagconf}
A \emph{flag configuration} of a big cohomology class $\theta$ is a (partial) flag of coherent analytic ideal sheaves
\[
\mathfrak{a}_0 \subseteq \mathfrak{a}_1 \subseteq \dots \subseteq \mathfrak{a}_N  = \mathcal{I}(V_\theta).
\]
By convention, $\mathfrak{a}_{\ell}:=0$ for $\ell\in \mathbb{Z}_{<0}$ and $\mathfrak{a}_{\ell}:=\mathfrak{a}_N$ if $\ell\in \mathbb{Z}_{\geq N}$.

A flag configuration will be sometimes conveniently denoted as an analytic coherent ideal sheaf on the product $X \times \mathbb C$: 
$$\mathfrak a := \mathfrak a_0 + \mathfrak a_1 s + \dots \mathfrak a_{N-1} s^{N-1} + \mathfrak a_N (s^N) \subseteq \mathcal O(X \times \mathbb C).$$
\end{definition}

For $r \in \mathbb N$ and $\lambda \in \mathbb{N}$ we introduce the following coherent sheaves associated with a flag configuration $\mathfrak a$:
\begin{equation}\label{eq: graded_ideal_def}
\mathfrak{a}_{r,\lambda} : = \sum_{\substack{\lambda_1 + \dots +\lambda_r = \lambda \\ \lambda_i\in \mathbb{N}}} \mathfrak{a}_{\lambda_1}\mathfrak{a}_{\lambda_2} \dots \mathfrak{a}_{\lambda_r}=\sum_{\substack{\sum_j \beta_j=r\\ \sum_j j\beta_j=\lambda\\ \beta_j\in \mathbb{N} \text{ for }j\in \mathbb{N}}} \prod_{j=0}^{\infty}\mathfrak{a}_{j}^{\beta_j}.
\end{equation}
We notice that $\lambda \mapsto \mathfrak{a}^r_\lambda$ is increasing and 
 \begin{equation}\label{eq:amul}
     \mathfrak{a}_{r,\lambda}  \cdot \mathfrak{a}_{r',\lambda'} \subseteq \mathfrak{a}_{r + r',\lambda+\lambda'}.
 \end{equation} 
 For $\lambda\in \mathbb{R}_{>0}\setminus \mathbb{Z}$, we formally set $\mathfrak{a}_{r,\lambda}=\mathfrak{a}_{r,\lfloor\lambda \rfloor}$. For $\lambda<0$, we set $\mathfrak{a}_{r,\lambda}=0$.

Returning to the case when $\{\theta\} = c_1(L)$,  to a flag configuration $\mathfrak{a}$, one can associate a filtration of $R(X,L)=\bigoplus_{r\in \mathbb{N}} H^0(X,L^{r})$, following ideas from \cite{Od13}:
\begin{equation}\label{eq: filtrate_flag_def}
\mathcal F^\lambda_{\mathfrak a} H^0(X,L^{r}) := H^0(X,L^{r} \otimes \mathfrak{a}_{r,-\lambda})\subseteq H^0(X,L^{r}).
\end{equation}

\begin{proposition}
For any flag configuration $\mathfrak{a}$ of $L$, 
$\mathcal F^\lambda_{\mathfrak a}$ defined in \eqref{eq: filtrate_flag_def} is a filtration on $R(X,L)$ in the sense of Definition~\ref{def:filtration}.
\end{proposition}
\begin{proof}
    We need to verify the conditions in Definition~\ref{def:filtration}. Condition~(i) and (ii) are obvious. We verify Condition~(iii), which says that for any $\lambda,\lambda'\in \mathbb{R}$, $r,r'\in \mathbb{N}$, we have
    \[
    \mathcal F^\lambda_{\mathfrak a} H^0(X,L^{ar})\mathcal F^{\lambda'}_{\mathfrak a} H^0(X,L^{ar'})\subseteq \mathcal F^{\lambda+\lambda'}_{\mathfrak a} H^0(X,L^{r+r'}).
    \]
    By definition, this amounts to 
    \[
    H^0(X,L^{r} \otimes \mathfrak{a}_{r,-\lambda})\cdot H^0(X,L^{r'} \otimes \mathfrak{a}_{r',-\lambda'})\subseteq H^0(X,L^{r+r'} \otimes \mathfrak{a}_{r+r',-(\lambda+\lambda')}).
    \]
    It suffices to argue
    \begin{equation}\label{eq:multipafilt}
    \mathfrak{a}_{r,-\lambda}\cdot \mathfrak{a}_{r',-\lambda'}\subseteq \mathfrak{a}_{r+r',-(\lambda+\lambda')}.
    \end{equation}
    When $\lambda>0$ or $\lambda'>0$ \eqref{eq:multipafilt} is trivial. So we may assume that
    $\lambda,\lambda'\leq 0$, and  \eqref{eq:multipafilt} follows from \eqref{eq:amul}.
    
    It remains to argue Condition~(iv) in Definition~\ref{def:filtration}. Namely, the filtration is linearly bounded. By definition, $\mathcal{F}^{\lambda}_{\mathfrak{a}}H^0(X,L^{r})=0$ for $\lambda>0$. So it suffices to show that there is $C>0$ so that 
    \begin{equation}\label{eq:H0atriv}
    H^0(X,L^{r}\otimes \mathfrak{a}_{r,Cr})=H^0(X,L^{r}).
    \end{equation}
    We claim that it suffices to take $C=N$. In this case,
    $
    \mathfrak{a}_{r,Nr}=\mathfrak{a} ^r_N=\mathcal{I}(V_{\theta})^r\supseteq \mathcal{I}(rV_{\theta})
    $
    by the multiplicativity of multiplier ideal sheaves \cite{DEL00}. As $V_{\theta}$ has minimal singularities, $H^0(X,L^r\otimes \mathcal{I}(rV_{\theta}))=H^0(X,L^r)$, so \eqref{eq:H0atriv} follows.
\end{proof}

\begin{definition}\label{def:PSray2}
The \emph{Phong--Sturm ray} $\{r^{L,\mathfrak{a}}_t\}_t \in \mathcal R^1_\theta$ associated to a flag configuration $\mathfrak a = \mathfrak a_0 + \mathfrak a_1 s + \dots \mathfrak a_{N-1} s^{N-1} + \mathfrak a_N (s^N)$ is the Phong--Sturm ray of the associated filtration $\mathcal F_{\mathfrak a}^\lambda$ (recall  Definition~\ref{def:PSray1}). When there is no risk of confusion, we simply write $\{r^\mathfrak a_t\}_t$.
\end{definition}

Careful readers might notice that in the ample case there is a minor difference in our definition of filtration associated to a flag ideal and the one in \cite{BHJ17}, when $L$ is ample. Though this is the case, we now point out that the associated Phong--Sturm ray is going to be the same, regardless what filtration one works with, ours or the one in \cite{BHJ17}.

Assume that $L$ is ample and $\mathfrak{a}$ is a flag configuration with $\mathfrak{a}_N=\mathcal{O}_X$. We assume that $L\otimes \mathfrak{a}_j$ is globally generated for all $j$.
As described in \cite{Od13,BHJ17} it is possible to associate to such $\mathfrak a$ a test configuration: let $p:\mathcal{X}\rightarrow X\times \mathbb{C}$ be the normalized blow up of $X\times \mathbb{C}$ with respect to flag ideal $\mathfrak{a}:=\mathfrak{a}_0+\mathfrak{a}_1s+\dots +\mathfrak{a}_N(s^N)$ and let $E$ be the exceptional divisor. Set $\mathcal{L}=p^*p_1^*L\otimes \mathcal{O}(-E)$, where $p_1:X\times \mathbb{C}\rightarrow X$ is the natural projection. Then with respect to the obvious $\mathbb{C}^*$-action, $(\mathcal{X},\mathcal{L})$ is a normal semi-ample test configuration of $(X,L)$ in the traditional sense. Let $\mathcal{F}_{(\mathcal{X},\mathcal{L})}$ denote the $\mathbb{Z}$-filtration of $R(X,L)$ induced by $(\mathcal{X},\mathcal{L})$, see \cite[Section~2.5]{BHJ17} for the precise definition. It is well-known that 
\begin{equation}\label{eq:filttestconf}
\mathcal{F}_{(\mathcal{X},\mathcal{L})}^{\lambda}H^0(X,L^r)=H^0(X,L^r\otimes \overline{\mathfrak{a}^r}_{-\lambda}),
\end{equation}
where $\overline{\mathfrak{a}^r}$ is the integral closure of $\mathfrak{a}^r$ and the subindex $-\lambda$ denotes the coefficient of $s^{-\lambda}$ in the expansion in $s$ (\cite[Proposition~2.21]{BHJ17}). We formally extend $\mathcal{F}_{(\mathcal{X},\mathcal{L})}$ to an $\mathbb{R}$-filtration by setting
\[
\mathcal{F}_{(\mathcal{X},\mathcal{L})}^{\lambda}H^0(X,L^r)=\mathcal{F}_{(\mathcal{X},\mathcal{L})}^{\lceil\lambda \rceil}H^0(X,L^r).
\]
From \eqref{eq:filttestconf}, we find immediately that
\begin{equation}\label{eq: filt_containment}
\mathcal{F}_{(\mathcal{X},\mathcal{L})}^{\lambda}H^0(X,L^r)\supseteq \mathcal{F}^{\lambda}_{\mathfrak{a}}H^0(X,L^r),
\end{equation}
and strict containment naturally occurs, however we now confirm that the associated Phong--Sturm rays do coincide. 

Indeed, let $\{\psi^{(\mathcal{X},\mathcal{L})}_{\tau}\}_{\tau}$ be the test curve defined by the filtration $\mathcal{F}_{(\mathcal{X},\mathcal{L})}$. Due to Theorem~\ref{thm:Imodelcurvenabij}(ii) it is enough to show that $\psi^{(\mathcal{X},\mathcal{L})}{}^{\An} = (\hat r^\mathfrak a)^\An$. This is confirmed by the next lemma.

\begin{lemma}\label{lma:PSanample}
Under the assumptions above, we have
\begin{equation}\label{eq:PSLeq}
(\hat r^\mathfrak a)^{\An}(v)=\psi^{(\mathcal{X},\mathcal{L})}{}^{\An}(v)= -\sigma(v)(\mathfrak{a}) =\max_j (-v(\mathfrak{a}_j)-j) \quad \text{for }v\in X^{\Div}.
\end{equation}
\end{lemma}
\begin{proof}
Due to Fekete's lemma we have that $\phi^{\mathcal F^{\mathfrak a}_\tau} \nwarrow \phi^{\mathcal F^{\mathfrak a}_{\tau,2^k}} \geq \phi^{\mathcal F^{\mathfrak a}_{\tau,1}}$. Due to global generatedness of $L\otimes \mathfrak a_{-\tau}$ we get that $-v(\phi^{\mathcal{F}}_{\tau}) \geq -v\left(\mathcal{F}^{\tau}H^0(X,L)\right)=-v\left(\mathfrak a_j\right), \ v \in X^\Div$. As a result, 
$$(\hat r^\mathfrak a)^{\An} \geq \max_j (-v(\mathfrak{a}_j)-j).$$
We argue the reverse direction. Due to \eqref{eq: filt_containment} we have that $\check{\psi}^{(\mathcal{X},\mathcal{L})}_t\geq r^\mathfrak a_t.$ Hence, it suffices to show that
\begin{equation}\label{eq:BBJNA}
     \max_j (-v(\mathfrak{a}_j)-j) \geq \check{\psi}^{(\mathcal{X},\mathcal{L}),\An}(v)
\end{equation}
for any $v\in X^{\Div}$. 
However,  we actually have equality here. Indeed, by  \cite[Theorem~9.2]{RWN14}  and \cite[Lemma~4.4]{BBJ21} we have that $\check{\psi}^{(\mathcal{X},\mathcal{L}),\An}(v)=\varphi_{\mathcal X,\mathcal L}(v)$, where $\varphi_{\mathcal X,\mathcal L}(v)$ is the non-Archimedean potential associated with $(
\mathcal X,\mathcal L)$.
By definition, $\varphi_{(\mathcal X,\mathcal L)(v)}=-\sigma(v)(E)$. However, the pushforward of $\mathcal O(-E)$ is just $\overline{\mathfrak a},$ the integral closure of $\mathfrak a$. Hence, $\max_j (-v(\mathfrak{a}_j)-j) =-\sigma(v)(\mathfrak a)=-\sigma(v)(\overline{\mathfrak a})=\check{\psi}^{(\mathcal{X},\mathcal{L}),\An}(v)$. Putting everything together, we obtain equality in \eqref{eq:BBJNA}, finishing the proof.
\end{proof}

Now we come back to the general situation where $L$ is only assumed to be big. In this case, we note that when all $L \otimes \mathfrak{a}_j$ are globally finitely generated, the first step in the proof of the above result implies that
\begin{equation}\label{eq:PSLlowerbd}
(\hat r^\mathfrak a)^{\An}(v) \geq \max_j (-v(\mathfrak{a}_j)-j) = \sigma(v)(\mathfrak a).
\end{equation}

Next we show that the procedure of \eqref{eq:ImPhiexp} associates a sequence of flag configurations to any geodesic ray: 

\begin{lemma} \label{lem:PhiinducePSH_1} 
Let $\phi\in \PSH(X,\theta)$ be a model potential with positive mass. Let $\Phi \in \textup{PSH}(X\times \Delta,p_1^* \theta)$ be the potential corresponding to the geodesic ray $\{u_t\}_t$ in $\PSH(X,\theta)$ emanating from $\phi$ with $\sup u_1\leq 0$. Given any $m \geq 0$, let 
$$\mathcal I(m\Phi) = \mathfrak{a}_0 +\mathfrak{a}_1 s+\dots+\mathfrak{a}_{N-1} s^{N-1}+\mathfrak{a}_N (s^N),$$
as in \eqref{eq:ImPhiexp}. Then we have $\mathfrak{a}_N = \mathcal I(m \phi)$.
\end{lemma}

\begin{proof} Recall \eqref{eq: sheaf_identity} that $\mathcal I(m\Phi)|_{X\times \Delta^*}=p_1^* \mathcal I(m\phi)$. This allows to regard $\mathfrak{a}_0 +\mathfrak{a}_1 s+\dots+\mathfrak{a}_{N-1} s^{N-1}+\mathfrak{a}_N (s^N)$ as an algebraic coherent ideal sheaf on $X\times \mathbb C$ by Lemma \ref{lma:GAGA}, so
\begin{flalign*}
\left(\mathfrak{a}_0 +\dots+\mathfrak{a}_{N-1} s^{N-1}+\mathfrak{a}_N (s^N)\right)|_{X\times \mathbb{C}^*}&=\left(\mathfrak{a}_0 +\dots+\mathfrak{a}_{N-1} s^{N-1}+\mathfrak{a}_N (s^N)\right)\otimes_{\mathbb{C}[s]} \mathbb{C}[s,s^{-1}]\\
&=p_1^*\mathfrak{a}_N.
\end{flalign*}
Comparing the fibers at $s=1$ we conclude that $\mathfrak{a}_N=\mathcal{I}(m\phi)$.
\end{proof}

We argue that for the converse of \eqref{eq:PSLlowerbd}, one does not even need global finite generation:

\begin{lemma}\label{lem: PS_ray_NA_formula}
Let $\mathfrak a = \mathfrak a_0 + \mathfrak a_1 s + \dots + \mathfrak a_{N-1} s^{N-1} + \mathfrak a_N (s^N)$ be a flag configuration of a big line bundle $L$. 
Then
\begin{equation}\label{eq:PSanupp}
(\hat r^\mathfrak a)^{\An}(v)\leq \max_j (-v(\mathfrak{a}_j)-j)=\sigma(v)(\mathfrak a).
\end{equation}
\end{lemma}
\begin{proof} 
We first reduce the general case to the case when $L$ is very ample and $L\otimes \mathfrak{a}_j$ is globally generated for all $j$. 

Suppose that $A$ is a very ample line bundle so that $L\otimes A$ is ample and such that $A\otimes L\otimes \mathfrak{a}_j$ is globally generated for all $j$. Such $A$ exists by \cite[Theorem~II.7.6]{Har} for example. Choose any $C\in\mathbb{N}_{>0}$. Let $\mathfrak{b}$ be the following flag configuration of $L\otimes A$:
\[
\mathfrak{b}_k=\left\{ 
\begin{aligned}
\mathfrak{a}_k,\quad & k=1,\dots,N;\\
\mathfrak{a}_N,\quad & k=N+1,\dots,N+C-1;\\
\mathcal{O}_X,\quad & k=N+C.
\end{aligned}
\right.
\]
So $\mathfrak{b}_k\supseteq \mathfrak{a}_k$ for all $k\in \mathbb{N}$.
After choosing a smooth positive metric on $A$, it is obvious that
\[
(\hat r^{L,\mathfrak a})^{\An}\leq (\hat r^{L \otimes A,\mathfrak b})^{\An}.
\]
If we managed to prove the result for $\mathfrak{b}$ and $L \otimes A$, then we would have
\[
(\hat r^{L,\mathfrak a})^{\An}\leq (\hat r^{L \otimes A,\mathfrak b})^{\An}\leq \max\left\{\max_{j=1,\dots,N} (-v(\mathfrak{a}_j)-j),-N-C \right\}.
\]
Letting $C\to \infty$, we conclude \eqref{eq:PSanupp}. Therefore, the Lemma is reduced to proving \eqref{eq:PSanupp} for ample bundles $L$. But this was proved in  Lemma~\ref{lma:PSanample}.
\end{proof}

We will also need the following general approximation result for geodesic rays, reminiscent of \cite[Theorem~3.19]{DX20}:

\begin{lemma}\label{lem: ray_approx} Let $\{\phi_t\}_t \in \mathcal R^1_\theta$ be a ray in a big class $\{\theta\}$.  There exists an increasing sequence of subgeodesic rays $\{\psi^j_t\}_t \subset \textup{PSH}(X,\theta)$, not necessarily emanating from $V_\theta$,  such that $\theta_{\psi^j_t} \geq \varepsilon_j \omega$  for some $\varepsilon_j \searrow 0$, and $P[\hat \psi^j_\tau] \nearrow P[\hat \phi_\tau] = \hat \phi_\tau$ a.e. for any $\tau \in \mathbb R$. 
\end{lemma}

\begin{proof} For any $\delta> 0$, due to $\tau$-concavity of test curves, the potential $\hat \phi_{\tau^+_\phi-\delta}$ has non-zero mass, so by \cite[Proposition 3.6]{DX21} there exists $v_\delta \in \textup{PSH}(X,\theta)$ such that $v_\delta \leq \hat \phi_{\tau^+_\phi-\delta}$ and $\theta_{v_\delta} \geq \varepsilon_\delta \omega$ for some $\varepsilon_\delta>0$.

Let $\{\eta^\delta_t\}_t$ be the subgeodesic associated with the test curve $\hat \eta ^\delta_\tau  := (1-\delta )\hat \phi_\tau + \delta v_\delta$ for $\tau \leq \tau^+_\phi-\delta$ and $\hat \eta ^\delta_\tau := -\infty$ for $\tau > \tau^+_\phi-\delta$.

Since $\hat \eta ^\delta_\tau \leq \hat \phi_\tau$ for all $\tau \in \mathbb R$, we get that $\eta^\delta_t \leq \phi_t$ and $\theta_{\eta^\delta_t} \geq \delta \theta_{v_\delta} \geq \delta \varepsilon_\delta \omega$ for all $t \geq 0$. 
Also, we have that 
$$\hat \phi_\tau = P [\hat \phi_\tau] \geq P [\hat \eta^\delta_\tau] \geq (1-\delta)P [\hat \phi_\tau] + \delta P[v_\delta] \geq (1-\delta) \hat \phi_\tau + \delta P[v_\delta].$$
Since $\sup_X P [v_\delta] = 0$, letting $\delta \to 0$ we obtain that $P [\hat \eta^\delta_\tau] \to \hat \phi_\tau$ in the $L^1$ topology.

However, $\delta \to \eta^\delta_t$ is not easily seen to be $\delta$-increasing. To address this, we introduce the  sequence of subgeodesics $\psi^j_t := \max_{k=1,\dots,j}\eta^{{1}/{k}}_t$, that satisfies the requirements of the lemma and is additionally increasing.
\end{proof}

Finally, we arrive at the main result of this section:

\begin{theorem}\label{prop:BBJmaxappbyflagconf}
Let $\{\phi_t\}_t \in \mathcal R^1_{\mathcal{I}}(X,\theta)$ be an $\mathcal{I}$-maximal geodesic ray with potential $\Phi \in \textup{PSH}(X \times \Delta,p_1^*\theta)$, normalized by $\sup_X \phi_1 = 0$. For $m \geq 0$, let 
\[
\mathcal I(2^m\Phi) = \mathfrak{a}^m_0 +\mathfrak{a}^m_1 s+\dots+\mathfrak{a}_{N_m-1}^m s^{N_m-1}+\mathfrak{a}_{N_m}^m (s^{N_m}) \subseteq \mathcal O_{X \times \Delta},
\]
as in \eqref{eq:ImPhiexp}. 
Let $\{\phi^m_t\}_t:=\big\{2^{-m}r^{{L^{2^m}},\mathcal I(2^m\Phi)}_t\big\}_t \subset \textup{PSH}(X,\theta)$ be a rescaled Phong--Sturm ray.
We have that $\phi^m_t \searrow \phi_t$ for all $t \geq 0$. In particular, $d_1^c(\{\phi^m_t\}_t, \{\phi_t\}_t) \to 0$.
\end{theorem}

\begin{proof} 
Let $\{\mathcal F^{\lambda,m}\}_\lambda$ be the filtrations of $R(X,L^{2^m})$ induced by the flag configurations $\{\mathfrak{a}^m_k\}_k$ of $L^{2^m}$, defined in \eqref{eq: filtrate_flag_def}. 

By the subadditivity of multiplier ideals, we have that $\mathcal I(2^{m+1}\Phi) \subseteq \mathcal I(2^m\Phi)^2$. In particular, $\mathfrak a^{m+1}_k \subseteq  \sum_j \mathfrak a^{m}_{j}\mathfrak a^{m}_{k-j}$.
As a result, $\mathfrak a^{m+1}_{r,\lambda} \subset \mathfrak a^{m}_{2r,\lambda}$ (recall \eqref{eq: graded_ideal_def}).

We obtain that $\mathcal F^{\lambda,m+1} H^0(X, L^{2^{m+1}r}) \subset \mathcal F^{\lambda,m} H^0(X, L^{2^{m} 2r})$. This implies that $\hat \phi^{m+1}_\tau \leq \hat \phi^{m}_\tau$, hence $\{\phi^m_t\}_t$ is indeed $m$-decreasing.

By Lemma~\ref{lem: PS_ray_NA_formula},
\begin{equation}\label{eq: non_archim_est}
    \hat \phi^{m,\An}(v) \leq  -\frac{1}{2^m} \min_j (v(\mathfrak a_j^m) + i)=-2^{-m}\sigma(v)(\mathcal{I}(2^m\Phi))
\end{equation}
for any $v \in X^\Div$.
On the other hand, by \cite[Lemma~B.4]{BBJ21}, for any $v\in X^{\Div}$, we have
\begin{equation}\label{eq:sigmavupp}
 -2^{-m}\sigma(v)(\mathcal{I}(2^m\Phi))\leq -\sigma(v)(\Phi)+2^{-m}(A_X(v)+1).
\end{equation}
The right-hand side of \eqref{eq:sigmavupp} is nothing but $\hat \phi^{\An}(v)+2^{-m}(A_X(v)+1)$, so we obtain that
\begin{equation}\label{eq:phimphi2m}
\hat \phi^{m,\An}(v) \leq  \hat \phi^{\An}(v)+2^{-m}(A_X(v)+1).
\end{equation}

Let $\{\psi^j_t\}_t \subset \textup{PSH}(X,\theta)$ be the sequence of subgeodesics from Lemma \ref{lem: ray_approx}, approximating $\{\phi_t\}_t$ from below. Let $\Psi^j \in \textup{PSH}(X \times \Delta, \pi^*\theta)$ be the potentials associated to the subgeodesics $\{\psi^j_t\}_t$. Then
\[
\mathcal I(2^m\Psi^j) = \mathfrak{a}^{m,j}_0 +\mathfrak{a}^{m,j}_1 s+\dots+\mathfrak{a}_{N_{m,\varepsilon}-1}^{m,j} s^{N_{m,j}-1}+\mathfrak{a}_{N_{m,j}}^{m,j} (s^{N_{m,j}}).
\]
We have that $\theta_{\psi^j_t} \geq \varepsilon_j \omega$. We claim that $L^{2^m} \otimes \mathfrak{a}^{m,j}_k$ is globally generated for $m \geq m_0(j)$. Indeed. Let $(G,h)$ be a hermitian ample line bundle on $X$ with curvature equal to $\eta>0$.

For $m_1(j)$ big enough, we have that $ \theta_{\psi^j_t} \geq \frac{1}{2^{m_1}} \eta$. As a result, we can apply  Lemma \ref{lem: Siu_global_uniform generatedness} to the coherent sheaf $\pi_1^* (L^{2^{m_1}} \otimes G^{-1}) \otimes \mathcal I (2^{m_1} \Psi^j)$ and conclude existence of $k>0$ such that $p_1^* (L^{2^m} \otimes G^{k-2^{m-m_1}}) \otimes \mathcal I (2^{m} \Psi^j)$ for all $m \geq m_2(j) \geq m_1(j)$.
Since $p_1^* G^{2^{m-m_1} - k}$ is globally generated for $m$ big enough, we get that $p_1^* (L^{2^m}) \otimes \mathcal I (2^{m} \Psi^j)$ is globally generated for all $m \geq m_0(j)$. The argument of \eqref{eq:ImPhiexp} now  yields the claim.

We fix $m \geq m_0(j)$. Then by \eqref{eq:PSLlowerbd}, for any $v \in X^\Div$ we find 
\[
-2^{-m} \sigma(v)(\mathcal I(2^m\Psi^j)) \leq 2^{-m} \big(\hat r^{L^{2^m},\mathcal I(2^m\Psi^j)}\big)^{\An}(v).
\]
Due to $\Psi^j \leq \Phi$ we observe that 
\[
2^{-m} \big(\hat r^{L^{2^m},\mathcal I(2^m\Psi^j)}\big)^{\An}(v)\leq 2^{-m} \big(\hat r^{L^{2^m},\mathcal I(2^m\Phi)}\big)^{\An}(v)=\big(\hat \phi^{m} \big)^\An(v).
\]
Comparing the last two inequalities and letting $m\to\infty$, by \cite[Lemma B.4]{BBJ21} and Proposition \ref{prop:psiandPhi} we find
\[
\big(\hat \psi^{j}\big)^\An(v) = -\sigma(v)(\Psi^j) = \lim_{m \to \infty} -\frac{1}{2^m} \sigma(v)(\mathcal I(2^m\Psi^j))\leq \lim_{m\to\infty} \big(\hat \phi^{m} \big)^\An(v).
\]
Due to Lemma \ref{lma:DC} and the last part of Lemma \ref{lem: ray_approx} we have that $\hat \psi^{j,\An}(v) \nearrow \hat \phi^\An$. Hence, letting $j\to\infty$ we find that
\begin{equation}\label{eq:phimphi3m}
\hat \phi^{\An}(v)\leq \lim_{m\to\infty} \hat \phi^{m,\An}(v) \leq \phi^{m,\An}(v).
\end{equation}
By Theorem \ref{thm:NAtoNAbij} we obtain that $\phi_t \leq \phi^m_t$. Let $\chi_t := \lim_t \phi^m_t \in \textup{PSH}(X,\theta), \ t >0$ be the $\mathcal I$-maximal limit ray. We have  $\phi_t \leq \chi_t \leq \phi^m_t$.

Letting $m \to \infty$ in \eqref{eq:phimphi2m} and  \eqref{eq:phimphi3m}, we arrive at $\chi^\An = \phi^\An.$ Using Theorem \ref{thm:NAtoNAbij} we conclude that $\{\chi_t\}_t = \{\psi_t\}_t$, finishing the proof.
\end{proof}

\section{Applications to Ding stability}

In this section, we fix a connected projective manifold $X$ of dimension $n$. Take a smooth closed real $(1,1)$-form $\theta$ on $X$ representing a big cohomology class $\{\theta\}$ and $V:=\int_X \theta_{V_{\theta}}^n$.

We consider a general qpsh function  $\psi$ on $X$. In addition, one can consider another qpsh function $\chi$ on $X$ with analytic singularity type. We will apply the results from previous sections to study the stability notions associted to the KE type equation \eqref{eq: KE_cont equation}.

In what follows, set $\mu:=e^{\chi-\psi}\omega^n$ to be the tame measure as in \cite{DZ22}. 
Let $\delta_\mu:=\delta_\mu(\{\theta\})$ be the delta invariant defined in \eqref{eq: delta_psi_def}. 
And for any $u\in \PSH(X,\theta)$, we set
$
c_\mu[u] := \sup\left\{\gamma \geq 0 : \int_X e^{-\gamma u} \,\mathrm{d}\mu <\infty\right\}.
$

As in  \cite{DZ22}, we will investigate the following energy functionals going back to \cite{Ding88}. For any $\lambda\in(0,c_\mu[V_\theta])$, set 
\begin{flalign*}
\mathcal{L}^\lambda_\mu(\varphi):=-\frac{1}{\lambda}\log\int_Xe^{-\lambda\varphi} \mathrm{d}\mu\text{ and }\mathcal D_\mu^\lambda(\varphi)=\mathcal{L}^\lambda_\mu(\varphi)-I_\theta(\varphi)\text{ for }\varphi\in\mathcal E^1(X,\theta),
\end{flalign*}
where $I_\theta(\cdot)$ is the Monge--Amp\`ere energy and $\mathcal D_{\mu}^\lambda$ is called the \emph{$\lambda$-Ding functional}.

Given any finite energy sublinear subgeodesic ray $\{u_t\}_t$, it is convenient to introduce radial functionals:
\begin{equation}
\label{eq:def-radial-functionals}
    \mathcal{L}^\lambda_\mu\{u_t\}:=\varliminf_{t \to \infty} \frac{\mathcal  L_{\mu}^\lambda(u_t)}{t}, I_\theta\{u_t\}:=\lim_{t \to \infty} \frac{I_\theta(u_t)}{t}\mathcal,\text{ and } \mathcal D_{\mu}^\lambda \{u_t\} :=\mathcal{L}^\lambda_\mu\{u_t\}-I_\theta\{u_t\}.
\end{equation}

Then one has the following useful radial formulae (see \cite[(8) and (25)]{DZ22}):
\begin{equation}
\label{eq:radial-Ding-formula}
\mathcal{L}_{\mu}^\lambda\{u_t\} = \sup\{\tau \ \in \RR \ : \  c_\mu[\hat u_\tau] \geq \lambda \}
\end{equation}
and
\begin{equation}
    \label{eq:radial-I-formula}
    I_\theta\{u_t\}=\tau^+_{\hat u}+\frac{1}{V}\int_{-\infty}^{\tau^+_{\hat u}}\bigg(\int_X\theta^n_{\hat u_\tau}-\int_X\theta^n_{V_\theta}\bigg)\mathrm{d}\tau.
\end{equation}

\subsection{Continuity of the singularity exponent and the radial Ding energy}

We first prove the continuity of the radial Ding functional.
Denote by  $\mathcal S_\theta$ the space of singularity types of $\theta$-psh functions. 

We start with the following lemma, complementing \cite[Theorem~1.2]{DK01}:
\begin{lemma}\label{lem: cont_sing_exp_decreas} Let $u_j,u \in \textup{PSH}(X,\theta)$ such that $u_j \searrow u$ and $\int_X \theta_{u_j}^n \searrow \int_X \theta_u^n >0$. Then $c_\mu[u_j] \to c_\mu[u]$.
\end{lemma}

\begin{proof} \cite[Lemma 4.3]{DDNL5} applies to give existence of $\alpha_j \searrow 0$ and $v_j \in \textup{PSH}(X,\theta)$ such that 
$(1-\alpha_j) u_j + \alpha_j v_j \leq u.$

Now let $c>0$ and $p>1$ such that $c < cp< \lim_j c_\mu[u_j]$.

Let $q>1$ be such that $1/p + 1/q =1$. By H\"older's inequality we have that 
\begin{flalign*}
\int_X e^{-c u } \,\mathrm{d}\mu \leq \int_X e^{-c ((1-\alpha_j) u_j + \alpha_j v_j) } \,\mathrm{d}\mu & \leq \| e^{-c ((1-\alpha_j) u_j}\|_{L^p(\mathrm{d}\mu)}\| e^{-c \alpha_j v_j}\|_{L^q(\mathrm{d}\mu)}\\
& \leq \| e^{-c u_j}\|_{L^p(\mathrm{d}\mu)}\| e^{-c \alpha_j v_j}\|_{L^q(\mathrm{d}\mu)}.
\end{flalign*}

As a result, to argue that $c < c_\mu[u]$, it is enough to argue that $\int_X e^{-c q \alpha_j v_j} \,\mathrm{d}\mu $ is finite for $j$ high enough.
By \cite[Theorem B.5]{BBJ21} there exists $r > 1$ be such that $\int_X e^{\chi - r \psi} \omega^n < \infty$. 

Let $t > 1$ such that $\frac{1}{t} + \frac{1}{r} =1$.
Another application of H\"older's inequality gives
$$\int_X e^{-c q \alpha_j v_j } \,\mathrm{d}\mu =\int_X e^{\chi -c q \alpha_j v_j  - \psi} \omega^n \leq C  \bigg(\int_X e^{\chi - c q t \alpha_j v_j}\omega^n \bigg)^{\frac{1}{t}}
\bigg(\int_X e^{\chi - r \psi}\omega^n \bigg)^{\frac{1}{r}}.$$

As $j \to \infty$ we have that $c \alpha_j q t \to 0$. As a result, the Lelong numbers of  $c \alpha_j q t v_j$ approach zero uniformly as $j \to \infty$. Therefore, by Skoda's integrability theorem we obtain that $\int_X e^{\chi - c q t \alpha_j v_j}\omega^n$ is finite for high enough $j$, implying that $c < c_\mu[u]$. Letting $c \nearrow \lim_j c_\mu[u_j]$, we arrive at $\lim_j c_\mu[u_j] \leq c_\mu[u]$. As the reverse inequality is trivial, the result follows.
\end{proof}

Recall that $\mathcal S_\theta$ admits a pseudometric $d_\mathcal S$, introduced and studied in \cite{DDNL5}. As a corollary of the above result, we note the following theorem, another complement to \cite[Theorem 1.2]{DK01}, as well as \cite[Theorem 1.3]{DDNL5}:

\begin{theorem}\label{thm: sing_type_cont} Let $[u_j],[u] \in \mathcal S_\theta$ such that $d_\mathcal S([u_j],[u]) \to 0$. Then $\lim_j c_\mu[u_j] = c_\mu[u]$.
\end{theorem}

\begin{proof} Notice that $u_j,u \in \textup{PSH}(X,\theta+\omega)$. By \cite[Corollary 4.3]{Xia21} $d^\theta_\mathcal S([u_j],[u]) \to 0$ implies that $d^{\theta+\omega}_\mathcal S([u_j],[u]) \to 0$. In particular, after replacing $\theta$ with $\theta + \omega$, we can assume that $\int_X \theta_{u}^n >0$.

By \cite[Proposition~2.5]{DZ22} we have that $c_\mu[u_j] = c_\mu[P[u_j]]$ and $c_\mu[u] = c_\mu[P[u]]$. By \cite[Theorem~3.3]{DDNL2}, we can replace $u_j,u$ and assume that $u_j = P[u_j]$ and $u = P[u]$, i.e., $u_j,u$ are model potentials.

Suppose that the conclusion is false, and we can find a subsequence of $[u_j]$, again denoted by $[u_j]$ such that $\lim_j c_\mu[u_j]$ exists but $\lim_j c_\mu[u_j] \neq c_\mu[u]$. 

By \cite[Theorem 5.6]{DDNL5}, after perhaps taking another subsequence, there exists $w_j,v_j \in \textup{PSH}(X,\theta)$ increasing/decreasing sequences such that $v_j \leq  u_j \leq w_j$ such that $d_\mathcal S([w_j],[u]) \to 0$,  $d_\mathcal S([v_j],[u]) \to 0$, $v_j \nearrow u$ and $w_j \searrow u$.

We have that $c_\mu[v_j] \leq c_\mu[u_j] \leq c_\mu[w_j]$ and $c_\mu[v_j] \to c_\mu[u]$ by Theorem \ref{thm:openness}. That $c_\mu[w_j] \to c_\mu[u]$ follows from the previous lemma, concluding the proof. 
\end{proof}

The following auxiliary result is also needed.

\begin{proposition} \label{prop: I_tau_d_1_c_conv} Let $\{u^k_t\}_t,\{u_t\}_t \in \mathcal R^1(X,\theta)$ such that $d_1^c(\{u^k_t\}_t,\{u_t\}_t) \to 0$. Then\vspace{0.1cm} \\
\noindent (i) $\tau^+_{\check u^k} \to \tau^+_{\check u}$ and $I_{\theta}\{u^k_t\} \to I_{\theta}\{u_t\}$. \vspace{0.1cm} \\
\noindent (ii) if $\{u^k_t\}_t$ is $k$-increasing/$k$-decreasing then $\hat u^k_\tau$ is increasing a.e./decreasing to $\hat u_\tau$, and $\int_X \theta^n_{\hat u^k_\tau} \to \int_X \theta^n_{\hat u_\tau}>0$, for any $\tau < \tau^+_{\check u}$.
\end{proposition}

\begin{proof}  Since $d_1^c(\{u^k_t\}_t,\{u_t\}_t) \to 0$, we have that 
$d_1(u^k_1,u_1) \to 0$. This implies that $I(u^k_1) \to I(u_1)$, giving $I\{u^k_t\} \to I \{u_t\}$. We also obtain that $\int_X |u^k_1 - u_1| \omega^n \to 0$, which implies that $\sup_X u^k_1 \to \sup_X u_1$, by Hartogs' lemma.

Since $\sup_X u^k_1 = \sup_X (u^k_1 - V_\theta) = \tau^+_{\check u^k}$ and $\sup_X u_1 = \sup_X (u_1 - V_\theta) = \tau^+_{\check u}$, we get that $\tau^+_{\check u^k} \to \tau^+_{\check u^k}$, proving (i).

Now we address (ii) in case  $\{u^k_t\}_t$ is $k$-decreasing. The increasing case, is handled similarly. Since $\tau^+_{\check u^k} \to \tau^+_{\check u^k}$, by \cite[Theorem 3.9]{DZ22} we have that
\begin{flalign*}
0 \leq I_{\theta}\{u^k_t\} - I_{\theta}\{u_t\} = \frac{1}{V}\int_{-\infty}^{\tau^+_{\hat u}} \left( \int_X\theta^n_{\hat u^k_\tau} - \int_X\theta^n_{\hat u_\tau}\right) \,\mathrm{d} \tau + o(k).
\end{flalign*}
As $\hat u^k_\tau \geq \hat u_\tau$, by \cite[Theorem~1.2]{WN19} we have that $\int_X\theta^n_{\hat u^k_\tau} \geq \int_X\theta^n_{\hat u_\tau}$ for $\tau < \tau^+_{\hat u}$. 

Since $I_{\theta}\{u^k_t\} \to I_{\theta}\{u_t\}$, by \cite[Lemma~3.11]{DZ22} we obtain that  $\int_X \theta^n_{\hat u^k_\tau} \to \int_X \theta^n_{\hat u_\tau}>0$ for all for $\tau < \tau^+_{\hat u}$. Finally, \cite[Theorem~3.12]{DDNL2} gives that $\hat u^k_\tau \searrow \hat u_\tau$ for all for $\tau < \tau^+_{\hat u}$. 
\end{proof}

\begin{theorem}\label{thm: rad_Ding_cont} 
For any $\lambda\in (0,c_{\mu}[V_{\theta}])$, the functional $\mathcal{D}_{\mu}^\lambda: \mathcal{R}^1(X,\theta) \to \mathbb{R}$ is continuous. 
\end{theorem}

\begin{proof} 
Suppose that $\{u^k_t\}_t,\{u_t\}_t  \in \mathcal R^1(X,\theta)$ for $k\in \mathbb{N}_{>0}$ such that $d_1^c(\{u^k_t\}_t,\{u_t\}_t) \to 0$. Then we need to show that 
$
\mathcal D_\mu^\lambda\{u_t^k\} \to \mathcal D_\mu^\lambda\{u_t\}.
$

Since $I_{\theta}\{u^k_t\} \to I_{\theta}\{u_t\}$, in light of \eqref{eq:radial-Ding-formula}, it is enough to argue that 
\begin{flalign}\label{eq: limit_to_argue}
\lim_{k\to\infty} \left( \sup\{\tau \in \RR  :   c_\mu[\hat u^k_\tau] \geq \lambda \}\right) = \sup\{\tau  \in \RR  :  c_\mu[\hat u_\tau] \geq \lambda \}. 
\end{flalign}

We first assume that $\{u^k_t\}_t$ is $k$-increasing. By Proposition \ref{prop: I_tau_d_1_c_conv}(ii) we get that $\hat u^k_\tau \nearrow \hat u_\tau$ a.e. for any $\tau < \tau^+_{\check u}$.
Clearly,
\[
\lim_{k\to\infty} \left( \sup\{\tau  \in \RR  :  c_\mu[\hat u^k_\tau] \geq \lambda \}\right) \leq \sup\{\tau  \in \RR  :   c_\mu[\hat u_\tau] \geq \lambda \}.
\]
Then the equality follows from Theorem \ref{thm:openness}, finishing the argument when $\{u^k_t\}_t$ is $k$-increasing.

Now we assume that $\{u^k_t\}_t$ is $k$-decreasing. By Proposition \ref{prop: I_tau_d_1_c_conv}(ii), $\hat u^k_\tau \searrow \hat u_\tau$ and $\int_X \theta_{\hat u^k_\tau}^n \searrow \int_X \theta_{\hat u_\tau}^n$ for any $\tau < \tau^+_{\check u}$.
As a result, $c_\mu[u^k_\tau] \to c_\mu[u_\tau]$ for any $\tau < \tau^+_{\check u}$, by Lemma \ref{lem: cont_sing_exp_decreas}. This implies that 
\[
\lim_{k\to\infty} \left( \sup\{\tau  \in \RR  :   c_\mu[\hat u^k_\tau] \geq \lambda \}\right) \leq \sup\{\tau  \in \RR  :   c_\mu[\hat u_\tau] \geq \lambda \},
\]
with the other direction being trivial.  Hence, the result follows when $\{u^k_t\}_t$ is $k$-decreasing.

For general $\{u^k_t\}_t$, by Proposition \ref{prop: conv_monotone}, after perhaps taking a subsequence, there exists $\{v^k_t\}_t, \{w^k_t\}_t$ $k$-increasing and $k$-decreasing rays respectively, such that $d_1^c(\{v^k_t\}_t,\{u_t\}_t) \to 0$,  $d_1^c(\{w^k_t\}_t,\{u_t\}_t) \to 0$ and $w^k_t \leq u^k_t \leq v^k_t$.

The latter condition implies that 
\[
\varliminf_{t\to\infty} \frac{-1}{\lambda t }\log \int_X e^{-\lambda w^k_t} \,\mathrm{d}\mu \leq \varliminf_{t\to\infty} \frac{-1}{\lambda t }\log \int_X e^{-\lambda u^k_t} \,\mathrm{d}\mu \leq \varliminf_{t\to\infty} \frac{-1}{\lambda t }\log \int_X e^{-\lambda v^k_t} \,\mathrm{d}\mu.
\]

As a result, by the first part of the argument we have that \eqref{eq: limit_to_argue} follows.
\end{proof}

\subsection{Ding stability in terms of filtrations and flag configurations}

The radial formulae \eqref{eq:radial-Ding-formula} and \eqref{eq:radial-I-formula} lead us to the following:
\begin{lemma}
\label{lem:radial-relation-under-maximization}
   Given a finite energy sublinear subgeodesic ray $\{u_t\}_t$, let $\{v_t\}_t$ be its maximization and $\{w_t\}_t$ its $\mathcal{I}$-maximization. Then one has
   $
   I_\theta\{w_t\}\geq I_\theta\{v_t\}=I_\theta\{u_t\},
   $
   $\mathcal{L}_{\mu}^\lambda\{w_t\}=\mathcal{L}_{\mu}^\lambda\{v_t\}=\mathcal{L}_{\mu}^\lambda\{u_t\}
   $
   and
   $
   \mathcal{D}_{\mu}^\lambda\{w_t\}\leq\mathcal{D}_{\mu}^\lambda\{v_t\}=\mathcal{D}_{\mu}^\lambda\{u_t\}.
   $
\end{lemma}

\begin{proof}
The first assertion follows from $w_t \geq v_t\geq u_t$ and \cite[(25)]{DZ22}. To show the second, it is enough to notice that $c_\mu[\hat v_\tau]=c_\mu[\hat u_\tau]$ (by \cite[Proposition 2.5]{DZ22}) and $c_\mu[\hat w_\tau]=c_\mu[\hat u_\tau]$ (by \cite[Theorem 2.3]{DZ22}, as $\hat w_\tau$ and $\hat u_\tau$ have the same Lelong number along any prime divisors over $X$ \cite{BFJ08}). The third follows from the previous two.
\end{proof}

This lemma implies the following characterization of $\delta_\mu$, refining \cite[Theorem 1.4]{DZ22}:

\begin{theorem}\label{thm: delta_I_maximal_semistab} When $\{\theta\}=c_1(L)$ for some big line bundle $L$, the following identities hold:
\begin{flalign}\label{eq: delta_geod_semi_stab}
\delta_\mu&=\sup\{\lambda>0 \ | \ \mathcal{D}_{\mu}^\lambda\{u_t\}\geq 0\text{ for all sublinear subgeodesic ray }u_t\in\mathcal{E}^1(X,\theta)\}\\
&=\sup\left\{\lambda>0 : \mathcal{D}_{\mu}^\lambda\{u_t\}\geq 0, \ \{u_t\}_t \in \mathcal R^1(X,\theta)\right\} \nonumber\\
&=\sup\left\{\lambda>0 : \mathcal{D}_{\mu}^\lambda\{u_t\}\geq 0, \ \{u_t\}_t \in \mathcal R^1_{\mathcal{I}}(X,\theta)\right\} \nonumber\\
&=\sup\left\{\lambda>0 : \mathcal{D}_{\mu}^\lambda\{u_t\}\geq 0, \ \{u_t\}_t \in \mathcal R^1_{\mathcal{I}}(X,\theta) \textup{ induced by filtrations}\right\} \nonumber\\
&=\sup\left\{\lambda>0 : \mathcal{D}_{\mu}^\lambda\{u_t\}\geq 0, \ \{u_t\}_t \in \mathcal R^1_{\mathcal{I}}(X,\theta) \textup{ induced by flag configurations}\right\}\nonumber
\end{flalign}
\end{theorem}
\begin{proof} The first identity of \eqref{eq: delta_geod_semi_stab} is just \cite[Theorem~1.4]{DZ22}. The next two identities follow from Lemma \ref{lem:radial-relation-under-maximization}. Since flag configurations induce a filtration, the equality of the last three lines follows from Theorem \ref{thm: rad_Ding_cont} 
\end{proof}

\subsection{Ding stability in terms of non-Archimedean data}

Let $\lambda\in(0,c_{\mu}[V_\theta])$. Assume that $\{\theta\}=c_1(L)$ for some big line bundle $L$ on $X$.

We define two non-Archimedean functionals: let $\{\phi_{\tau}\}_{\tau}\in \mathcal{E}^{1,\NA}(X,\theta)$, we set
\begin{equation}\label{eq:LNAdefinition}
L^{\lambda,\NA}_\mu\{\phi_{\tau}\}:=\adjustlimits\inf_{v\in X^{\Div}}\sup_{\tau\in \mathbb{R}}(\tau+A_{\chi,\psi}(v)-\lambda v(\phi_{\tau}))
\end{equation}
In terms of the associated non-Archimedean potential $\phi^{\An}$, recalling \eqref{eq:psiandef}, we can express $L^{\lambda,\NA}_\mu$ as follows:
\begin{equation}\label{eq:Lreformulate}
L^{\lambda,\NA}_\mu(\phi^{\An}) := L^{\lambda,\NA}_\mu\{\phi_{\tau}\}=\inf_{v\in X^{\Div}}(A_{\chi,\psi}(v)+ \phi^{\An}(\lambda v)).
\end{equation}
The non-Archimedean $\lambda$-Ding functional can be introduced as follows (recall \eqref{eq: INAdef}):
\begin{equation}\label{eq: Ding_NA_def}
\mathcal{D}^{\lambda,\NA}_{\mu}\{\phi_{\tau}\}:=L^{\lambda,\NA}_\mu\{\phi_{\tau}\}-I_{\theta}^{\NA}\{\phi_{\tau}\}.
\end{equation}

When $\lambda=1$ and $\{\theta\}$ is the first Chern class of an ample line bundle, these functionals correspond to the non-Archimedean functionals $L$ and $D$ defined in \cite[Definition~3.4]{BBJ21}.

We show that $L^{\lambda,\NA}_\mu$ and $\mathcal{D}^{\lambda,\NA}_\mu$ agree with the corresponding radial functionals \eqref{eq:def-radial-functionals}:

\begin{theorem}\label{thm:L=inf-sup}
For any $\{u_t\}\in\mathcal{R}^1_{\mathcal{I}}(X,\theta)$, we have
\begin{equation}\label{eq:Llambdamuut}
L^\lambda_\mu\{u_t\}=L^{\lambda,\NA}_\mu\{\hat{u}_{\tau}\},\quad \mathcal{D}^\lambda_\mu\{u_t\}=\mathcal{D}^{\lambda,\NA}_\mu\{\hat{u}_{\tau}\}
\end{equation}
\end{theorem}

\begin{proof}
We first observe that the second equality in \eqref{eq:Llambdamuut} follows from the first and Theorem~\ref{thm:RWNpresc}(iii).

Observe that in order to prove the first equation in \eqref{eq:Llambdamuut}, it suffices to prove it with $X^{\Div}_{\mathbb{R}}$ in place of $X^{\Div}$ in \eqref{eq:LNAdefinition}, where $X^{\Div}_{\mathbb{R}}$ denotes the set of $c\ord_E$ for all $c\in \mathbb{R}_{>0}$ and $E$ is an arbitrary prime divisor over $X$. 
In other words, it suffices to show that
\[
L^\lambda_\mu\{u_t\}=\inf_{v\in X_{\mathbb{R}}^{\Div}}\sup_{\tau\in \mathbb{R}}(\tau+A_{\chi,\psi}(v)-\lambda v(\phi_{\tau})).
\]
In fact, by \eqref{eq:Lreformulate}
\[
L^{\lambda,\NA}_\mu(\phi)=\inf_{v\in X^{\Div}}(A_{\chi,\psi}(v)+ \phi(\lambda v)).
\]
Here $\phi$ denotes $u^{\An}$, which lies in $\PSH(L^{\An})$ by Corollary~\ref{cor:testcurvegivepshme}. But then, as $\phi$ is usc, it follows that
\[
\inf_{c\in \mathbb{Q}_{>0}} (A_{\chi,\psi}(c\ord_E)+\phi(\lambda c \ord_E))=\inf_{c\in \mathbb{R}_{>0}} (A_{\chi,\psi}(c\ord_E)+\phi(\lambda c \ord_E))
\]
for all prime divisors over $X$. It follows that replacing $X^{\Div}$ by $X^{\Div}_{\mathbb{R}}$ on the right-hand side of \eqref{eq:LNAdefinition}, we end up with the same quantity.

For any $\tau_0< L^\lambda_\mu\{u_t\}$ we have by \eqref{eq:radial-Ding-formula} and \cite[Theorem~2.2]{DZ22} that $A_{\chi,\psi}(v)\geq\lambda v(\hat{u}_{\tau_0})\text{ for any }v\in X^{\Div}_\mathbb R$. 
Thus,
\[
\inf_{v\in X^{\Div}_\mathbb R}\sup_{\tau\in \mathbb{R}}\{\tau+A_{\chi,\psi}(v)-\lambda v(\hat u_\tau)\}\geq\inf_{v\in X^{div}_\mathbb R}\{\tau_0+A_{\chi,\psi}(v)-\lambda v(\hat{u}_{\tau_0})\}\geq\tau_0,
\]
so that
\[
L^\lambda_\mu\{u_t\}\leq \inf_{v\in X^{\Div}_\mathbb R}\sup_{\tau\in \mathbb{R}}\{\tau+A_{\chi,\psi}(v)-\lambda v(\hat u_\tau)\}.
\]

To see the reverse direction, we first consider the case when $L^\lambda_\theta\{u_t\}=\tau^+_{\hat{u}}$, which we will always assume to be $0$ after adding a linear term to $u_t$. In this scenario, it is clear that
\[
\sup_{\tau\in \mathbb{R}}\{\tau+A_{\chi,\psi}(v)-\lambda v(\hat u_\tau)\}\leq A_{\chi,\psi}(v).
\]
Note that $A_{\chi,\psi}(v)$ is always positive, which however can be made as small as we want simply by rescaling $v$. So we obtain that
\[
\inf_{v\in X^{\Div}_\mathbb R}\sup_{\tau\in \mathbb{R}}\{\tau+A_{\chi,\psi}(v)-\lambda v(\hat u_\tau)\}\leq 0=L^\lambda_\mu\{u_t\}.
\]

In what follows we assume that $L^\lambda_\mu\{u_t\}<\tau^+_{\hat{u}}=0.$
Take any $\tau_0\in(L^\lambda_\mu\{u_t\},0)$. Using \cite[Theorem~2.2]{DZ22}  and \eqref{eq: Ding_slope} we deduce that $c_{\mu}[\hat u_{\tau_0}]<\lambda$. So by \cite[Theorem~B.5]{BBJ21} there exists a prime divisor $E$ over $X$ such that $A_{\chi,\psi}(E)-\lambda \nu(\hat{u}_{\tau_0},E)\leq0.$
Now for any $a\in\mathbb{R}_{>0}$ consider the functions
\[
f_a(\tau):=\tau+aA_{\chi,\psi}(E)-a\lambda\nu(\hat{u}_\tau,E)\text{ and }g_a(\tau):=aA_{\chi,\psi}(E)-a\lambda \nu(\hat u_\tau,E)
\]
for $\tau\in(-\infty,0]$. Both are concave when $\tau\in(-\infty,0)$ and $f_a(0)\leq\lim_{\tau\rightarrow 0^-}f_a(\tau).$ Moreover, for sufficiently large $A>0$ we have $c_{\mu}[\hat u_{-A}]>\lambda$ (since $c_{\mu}[\hat u_\tau]\nearrow c_{\mu}[V_\theta]>\lambda$ as $\tau\rightarrow-\infty$ by Theorem \ref{thm:openness}), so that $g_a(-A)=aA_X(E)-a\lambda \nu(\hat u_{-A},E)>0$ by \cite[Theorem~B.5]{BBJ21}. This implies that
\[
g^\prime_a(\tau_0^-)=\lim_{h\rightarrow 0^+}\frac{g_a(\tau_0)-g_a(\tau_0-h)}{h}\leq\frac{a\lambda(\nu(\hat u_{-A},E)-\nu(\hat u_{\tau_0},E))}{\tau_0+A}<0.
\]
Now choose suitable $a_0\in\RR_{>0}$ such that $g^\prime_{a_0}(\tau_0^-)=-1.$ Then we have
$f_{a_0}^\prime(\tau^-_0)=1-1=0$,
from which we deduce that
\[
\sup_{\tau\leq 0} f_{a_0}(\tau)=f_{a_0}(\tau_0)=\tau_0+a_0(A_{\chi,\psi}(E)-\lambda\nu(\hat u_{\tau_0},E))\leq\tau_0.
\]
Thus, we arrive at
\[
\adjustlimits\inf_{v\in X^{\Div}_{\mathbb{R}}}\sup_{\tau\leq 0}\{\tau+A_{\chi,\psi}(v)-\lambda v(\hat u_\tau)\}\leq\tau_0,
\]
which then completes the proof.
\end{proof}
As a consequence of the proceeding theorem and Theorem \ref{thm: delta_I_maximal_semistab}, we can add one more equality to Theorem~\ref{thm: delta_I_maximal_semistab}, finishing the proof of Theorem \ref{mthm: delta}:
\begin{corollary}\label{cor: NAcor}  Assume that $\{\theta\}=c_1(L)$ for some big line bundle $L$ on $X$.
    We have
    \[
        \delta_{\mu}=\sup\left\{\lambda>0 : \mathcal{D}_{\mu}^{\lambda,\NA}(u) \geq 0, \ u \in \mathcal E^{1,\NA}(X,\theta)\right\}.
    \]
\end{corollary}

\subsection{A Yau--Tian--Donaldson type existence theorem}\label{sec: YTD}

In this subsection we focus on the case $L = -K_X$. Also, for simplicity, we assume that there is no twisting, so  $\psi = 0$ and $\chi = f$, where  $\chi :=f \in C^\infty(X)$ is a Ricci potential, satisfying $\theta + \ddc f = \Ric \omega$.

We recall the definition of uniform Ding stability. We say that $(X,-K_X)$ is \emph{uniformly Ding stable with respect to flag configurations}, if there exists $\varepsilon >0$ such that
\begin{equation}\label{eq: uniform_Ding_stab_flag_def}
\mathcal{D}_\mu^{1,\NA} (u)\geq \varepsilon \mathcal J^{\NA}(u),
\end{equation}
for any $u = \{\hat u_\tau\}_\tau \in \mathcal E^{1,\NA}(X,\theta) $ induced by   flag configurations. By Theorem \ref{thm:L=inf-sup} we have that $L^{1,\NA}_\mu\{\hat u_{\tau}\} = L^1_\mu\{u_t\} =\inf_{v\in X^{\Div}}(A_X(v)+ u^{\An}(v))$.

Bringing together \eqref{eq: Ding_NA_def} and  \eqref{eq:Jreformulate}, the  condition \eqref{eq: uniform_Ding_stab_flag_def} can be reformulated in a non-Archimedean/algebraic language in the following manner:
\begin{equation}\label{eq: Ding_algebraic}
\inf_{v\in X^{\Div}}(A_X(v)+ u^{\An}(v))  \geq \varepsilon \tau^+_u +  (1 -\varepsilon) I_\theta^\NA(u),
\end{equation}
for all $u = \{\hat u_\tau\}_\tau\in \mathcal E^{1,\NA}(X,\theta)$ induced by flag configurations. Using Proposition \ref{prop:I=S}, the quantities on the right hand side of this inequality can be computed in terms of the filtration $\mathcal F_u$ of the flag configuration. Indeed, $\tau^+_u = \tau_L(\mathcal F_u)$ and $I_\theta^\NA(u) = S_L(\mathcal F_u)$,  allowing for an algebraic/valuative interpretation of uniform Ding stability.

Finally, we prove our last main result, a YTD type existence theorem for KE metrics:

\begin{theorem}(=Theorem \ref{thm:YTD}) \label{thm:YYTD}Suppose that $L=-K_X$ is big. If $(X,-K_X)$ is uniformly Ding stable with respect to flag configurations then there exists a K\"ahler--Einstein metric, i.e., there exists a solution to the following equation, having minimal singularity type:
\begin{equation}\label{eq: KE_eq}
\theta_u^n = e^{f- u} \omega^n, \ \ \ u \in \textup{PSH}(X,\theta),
\end{equation}
with $f\in C^\infty(X)$ satisfying $\theta+\ddc f=\Ric \omega$.
\end{theorem}

\begin{proof} The argument is by contradiction. Put $\mu:=e^{f}\omega^n$. Suppose that \eqref{eq: KE_eq} does not have a solution with minimal singularity type. Then by \cite[Theorem 5.3]{DZ22} there exists $\{u_t\}_t \in \mathcal R^1(X,\theta)$ such that $\mathcal D^1_\mu\{u_t\} \leq 0$, $\sup_X u_t =0$ and $I\{u_t\} = -1$. Let $\{w_t\} \in \mathcal R^1_\mathcal I(X,\theta)$ be the $\mathcal I$-maximization of $\{u_t\}_t$. By Lemma \ref{lem:radial-relation-under-maximization} we get that $\mathcal L^1_\mu\{w_t\} =\mathcal L^1_\mu\{u_t\}$ and $I\{w_t\} \geq I\{u_t\}$.

Hence, either $u_t = w_t$ and $\mathcal D^1_\mu\{w_t\} \leq 0$. Or, $w_t \neq u_t$ and $\mathcal D^1_\mu\{w_t\}< 0$. Hence, in both cases $\{w_t\}_t$ must be a non-trivial $\mathcal I$-maximal geodesic ray. 

After reparametrizing $\{w_t\}_t$ so that $I\{w_t\} = -1$, we still have that $\mathcal D^1_\mu\{w_t\} \leq 0$.
Theorem \ref{prop:BBJmaxappbyflagconf} gives $w^k = \{w^k_t\}_t \in \mathcal R^1_\mathcal I(X,\theta) = \mathcal E^{1,\NA}(X,\theta)$, a sequence of rays induced by flag configurations such that $d_1^c(\{w^k_t\},\{w_t\}) \to 0$. Due to $d_1^c$-convergence, we have that $\tau^+_ {\hat w^k} =\sup_X w^k_1 \to \tau^+_{\hat w} = \sup_X w_1 = 0$ and $I\{w^k_t\}  \to I\{w_t\}= -1$. In particular, $\mathcal J^{\NA}(w^k) \to 1$ (recall \eqref{eq:Jreformulate}).

By Theorem \ref{thm: rad_Ding_cont}, $\mathcal D^1_\mu\{w^k_t\}  \to \mathcal D^1_\mu\{w_t\} \leq 0$, contradicting uniform Ding stability.
\end{proof}

\begingroup
\setstretch{1.1}
\setlength\bibitemsep{0pt}
\setlength\biblabelsep{0pt}
\printbibliography
\endgroup

\small
\noindent {\sc Department of Mathematics, University of Maryland, 
4176 Campus Dr, College Park, MD 20742, USA}\\
{\tt tdarvas@umd.edu}\vspace{0.1in}

\noindent {\sc Institut de Mathématiques de Jussieu-Paris Rive Gauche, 4 Place Jussieu, Paris, 75005, France}\\
{\tt mingchen@imj-prg.fr}\vspace{0.1in}

\noindent {\sc School of Mathematical Sciences, Beijing Normal University, 19 Xinjiekou Wai Street, Beijing 100875, China}\\
{\tt kwzhang@bnu.edu.cn}
\end{document}